\documentclass[reqno,11pt]{article}
\usepackage{amsmath, latexsym, amsfonts, amssymb, amsthm, amscd}
\usepackage{mathrsfs}
\usepackage[utf8]{inputenc} 
\usepackage[english]{babel}
\usepackage{hyperref}
\usepackage{graphicx}
\usepackage{array}
\usepackage{booktabs}
\usepackage{mathtools}
\usepackage{subfig}

\usepackage{tabu}
\usepackage{geometry}
\usepackage{url}
\usepackage[backend=bibtex,style=alphabetic,firstinits=true,
            doi=false,isbn=false,url=false]{biblatex}
\newcommand{\blackboard}[1]{\mathbf{#1}} 
\newcommand{\N}{\blackboard{N}}%		entiers
\newcommand{\PB}{\blackboard{P}}%	 	probabilite
\newcommand{\EB}{\blackboard{E}}% 		esperance
\newcommand{\R}{\blackboard{R}}

\newcommand{\calig}[1]{\mathcal{#1}}
\newcommand{\A}{\calig{A}}% 			generateur infinitesimal diff
\renewcommand{\d}[1]{\textup{d} #1}% 	le dx de l'integrale
\newcommand{\convloi}{%				convergence en loi
    \xrightarrow[\phantom{ a.s }]{\calig{L}}}

\newcommand{\delimleft}[2]{\ifcase #1\or%	delimiteur gauche
    \bigl#2\or %
    \Bigl#2\or %
    \biggl#2\or %
    \Biggl#2\or %
    \left#2\fi}
\newcommand{\delimright}[2]{\ifcase #1\or%	delimiteur droite
    \bigr#2\or %
    \Bigr#2\or %
    \biggr#2\or %
    \Biggr#2\or %
    \right#2\fi}
\newcommand{\pa}[2][5]{% 			parentheses
    \delimleft{#1}{(} #2 \delimright{#1}{)}}
\newcommand{\br}[2][5]{%	 		crochets
    \delimleft{#1}{[} #2 \delimright{#1}{]}}
\newcommand{\ac}[2][1]{%			accolades
    \delimleft{#1}{\{} #2 \delimright{#1}{\}}}

\newcommand{\abs}[2][1]{%			valeur absolue
    {\delimleft{#1}{\lvert} #2%
    \delimright{#1}{\rvert}}}
\newcommand{\norm}[2][1]{% 			norme
    {\delimleft{#1}{\lVert} #2%
    \delimright{#1}{\rVert}}}
\newcommand{\normLp}[3][1]{%			norme L^p
    \norm[#1]{#3}_{\scriptscriptstyle #2}}

\newcommand{\normal}[2][5]{%			loi normale
    \calig{N}\! \pa[#1]{#2}}

\newcommand{\ind}[2][1]{%			indicatrice
    \boldsymbol{1}_{\ac[#1]{#2}}}

\newcommand{\esp}[2][5]{%			esperance
    \EB \br[#1]{#2}}

\newcommand{\espc}[3][5]{%			esperance conditionnelle
    \EB \br[#1]{#2%
    \delimleft{#1}{.} \vphantom{#2}\vphantom{#3}%
    \delimright{#1}{\vert} #3}}

\renewcommand{\le}{\leqslant}% 			<= plus beau
\renewcommand{\ge}{\geqslant}%			>= plus beau
\newcommand{\ds}{\displaystyle}%		displaystyle
\DeclareMathOperator*{\argmin}{argmin}

\newtheorem{Theorem}{Theorem}[section]

\newtheorem{Proposition}[Theorem]{Proposition}

\newtheorem{Lemma}[Theorem]{Lemma}

\newtheorem{Remark}[Theorem]{Remark}

\DeclareMathOperator{\PP}{\blackboard{P}}

\DeclareMathOperator{\Var}{Var} 
\DeclareMathOperator{\RR}{\blackboard{R}}
\DeclareMathOperator{\Cost}{Cost}

\DeclareMathOperator{\effort}{Effort}

\DeclareMathOperator{\W}{\mathbf{W}}
\DeclareMathOperator{\w}{\mathbf{w}}
\DeclareMathOperator{\Hr}{\calig{H}}

\DeclareMathOperator{\Lip}{Lip}

\newcommand{\MLMC}{MLMC}
\newcommand{\MLRR}{ML2R}

\newcommand{\iid}{\emph{i.i.d.}\,}

\newcommand{\calL}{\mathcal{L}}
\newcommand{\calN}{\mathcal{N}}
\newcommand{\hh}{\mathbf{h}}

\hypersetup{pdfborder={0 0 0},colorlinks,citecolor=blue, urlcolor=blue}

\providecommand{\ceil}[1]{\left \lceil #1 \right \rceil }
\providecommand{\abs}[1]{\left| #1 \right| }

\title{Limit theorems for weighted and regular Multilevel estimators}

\author{Daphn\'e Giorgi\footnote{Laboratoire de Probabilit\'es et Mod\`eles Al\'eatoires, UMR 7599, UPMC Paris 6 (Sorbonne Universit\'e), E-mail: \url{daphne.giorgi@upmc.fr}},
Vincent Lemaire\footnote{Laboratoire de Probabilit\'es et Mod\`eles Al\'eatoires, UMR 7599, UPMC Paris 6 (Sorbonne Universit\'e), 
E-mail: \url{vincent.lemaire@upmc.fr}}, Gilles Pag\`es\footnote{Laboratoire de Probabilit\'es et Mod\`eles Al\'eatoires, UMR 7599, UPMC Paris 6 (Sorbonne Universit\'e), E-mail: \url{gilles.pages@upmc.fr}}
}

\begin{document}

\maketitle
  
\begin{abstract}
	We aim at analyzing in terms of $a.s.$ convergence and weak rate the performances of the Multilevel Monte Carlo estimator (MLMC) introduced in~\cite{Gi08} and of its weighted version, the Multilevel Richardson Romberg estimator (ML2R),  introduced  in~\cite{LePa14}.
	These two estimators permit to compute a very accurate approximation of $I_0 = \esp{Y_0}$ by a Monte Carlo type estimator when the (non-degenerate) random variable $Y_0 \in L^2(\PB)$ cannot be simulated (exactly) at a reasonable computational cost whereas a family of simulatable approximations $(Y_h)_{h \in \Hr}$ is available. 
	We will carry out these investigations in an abstract framework before applying our results, mainly a Strong Law of Large Numbers
	and a Central Limit Theorem, to some typical fields of applications: discretization schemes of diffusions and nested Monte Carlo.
\end{abstract}

\section{Introduction}
In recent years, there has been an increasing interest in Multilevel Monte Carlo approach which delivers remarkable improvements in computational complexity in comparison with standard Monte Carlo in biased framework. We refer the reader to~\cite{Gi15} for a broad outline of the ideas behind the Multilevel Monte Carlo method and various recent generalizations and extensions.
In this paper we establish a Strong Law of Large Numbers and Central Limit Theorem for two kinds of multilevel estimator, 
Multilevel Monte Carlo estimator (MLMC) introduced by Giles in~\cite{Gi08} 
and the  Multilevel Richardson-Romberg (weighted) estimator introduced in~\cite{LePa14}. 
We consider a rather general and in some way abstract framework which will allow us to state these results 
whatever the strong rate parameter is (usually denoted by $\beta$). 
To be more precise we will deal with the versions of these estimators designed to achieve a root mean squared error (RMSE) $\varepsilon$ and establish these results as $\varepsilon \to 0$.
Doing so we will retrieve some recent results established in~\cite{BeKe15} in the framework 
of Euler discretization schemes of Brownian diffusions. 
We will also deduce a SLLN and a CLT for Multilevel nested Monte Carlo, 
which are new results to our knowledge. 
More generally our result apply to any implementation of Multilevel Monte Carlo methods.

Let $(\Omega, \A, \PB)$ be a probability space and $(Y_h)_{h \in \Hr}$ be a family of real-valued random variables in $L^2(\PB)$ associated to $Y_0$ where $\Hr = \left\{\frac{\hh}{n}, n\geq1\right\}$ such that $\lim_{h \to 0} \normLp{2}{Y_h-Y_0} = 0$. In the sequel, a fixed $h \in \Hr$ will be called \emph{bias parameter} (though it appears in a different framework as a discretization parameter).
In what follows we will be interested in the computational cost of the estimators denoted by the $\Cost(\cdot)$ function.
We assume that the simulation of $Y_h$ has an inverse linear complexity \emph{i.e.} $\Cost(Y_h) = h^{-1}$. 
A natural estimator of $I_0=\esp{Y_0}$ is the standard Monte Carlo estimator, which reads for a fixed $h$  
\begin{equation}
	\label{eq::crude-estimator}
	I_h^N = \frac 1 N \sum_{k=1}^N Y_h^{k} \quad \text{with} \quad \Cost(I_h^N) = h^{-1} N,
\end{equation}
where $(Y_h^k)_{k\geq1}$ are \iid copies of $Y_h$ and $N$ is the size of the estimator, which controls the statistical error. 
In order to give the definition of a Multilevel estimator, we consider a \emph{depth} $R \ge 2$ (the finest level of simulation) and a geometric decreasing sequence of bias parameters $h_j=h/n_j$ with $n_j = M^{j-1}$, $j=1,\ldots, R$. 
If $N$ is the estimator size, we consider an allocation policy $q=(q_1,\ldots, q_R)$, such that, at each level $j=1, \ldots, R$, we will simulate $N_j = \ceil{N q_j}$ scenarios (see \eqref{intro:MLMC-estimator} and \eqref{intro:MLRR-estimator} below). 
Thus, we consider $R$ independent copies of the family $Y^{(j)}=(Y_{h}^{(j)})_{h\in\mathcal H}$, $j=1,\ldots, R$, attached to $independent$ random copies $Y_0^{(j)}$ of $Y_0$. Moreover, let $(Y^{(j),k})_{k\geq 1}$ be independent sequences of independent copies of $Y^{(j)}$.
We denote by $I_\pi^N$ an estimator of size $N$ of $I_0$, attached to a simulation parameter $\pi\in\Pi\subset\R^d$.

$\rhd$ A standard Multilevel Monte Carlo (MLMC) estimator, as introduced by Giles in~\cite{Gi08}, reads 

\begin{equation}\label{intro:MLMC-estimator}
	I_\pi^N = I^N_{h,R,q} = \frac{1}{N_1} \sum_{k=1}^{N_1} Y_{h}^{(1),k} + \sum_{j=2}^R \frac{1}{N_j} \sum_{k=1}^{N_j} \pa{Y_{h_j}^{(j),k} - Y_{h_{j-1}}^{(j),k}}
\end{equation}
with $\pi = (h, R, q)$. 

$\rhd$ A Multilevel Richardson Romberg (ML2R) estimator, as introduced in~\cite{LePa14}, is a weighted version of \eqref{intro:MLMC-estimator} which reads
\begin{equation}\label{intro:MLRR-estimator}
	I_\pi^N = I^N_{h,R,q} = \frac{1}{N_1} \sum_{k=1}^{N_1} Y_{h}^{(1),k} + \sum_{j=2}^R \frac{\W_j^R}{N_j} \sum_{k=1}^{N_j} \pa{Y_{h_j}^{(j),k} - Y_{h_{j-1}}^{(j),k}}
\end{equation}
with $\pi = (h,R,q)$. The weights $\left(\W_j^R\right)_{j=1, \ldots, R}$ are explicitly defined as functions 
of the weak error rate $\alpha$ (see equation \eqref{lemaireweak_error} below) and of the refiners $n_j$, $j=0,\ldots, R$ in order to kill the successive bias terms in the weak error expansion (see Section \ref{section::weights} for more details on the weights).
When no ambiguity, we will keep denoting by $I_\pi^N$ estimators for both classes.
We notice that a  Crude Monte Carlo estimator of size $N$ formally appears as an ML2R estimator with $R=1$ and
a MLMC estimator appears as an ML2R estimator in which the weights set $\W_j^R = 1$, $j=1, \ldots, R$. Based on the inverse linear complexity of $Y_h$, it is clear that the simulation cost of both MLMC and ML2R estimators is given by 
$$\Cost(I^N_{h,R,q}) = \frac N h \sum_{j=1}^R q_j (n_{j-1} + n_j)$$
with the convention $n_0=0$. The difference between the cost of MLMC and of ML2R estimator comes from the different choice of the parameters $M$, $R$, $h$, $q$ and $N$.

The calibration of the parameters is the result, a root $M \ge 2$ being fixed, of the minimization of the simulation cost, for a given target Mean Square Error or $L^2$-error $\varepsilon$, namely,
\begin{equation}
 \label{eq::cost_min}
 (\pi(\varepsilon), N(\varepsilon)) = \argmin_{\normLp{2}{I_\pi^N - I_0} \leq \varepsilon} \Cost(I_\pi^N).
\end{equation}
This calibration has been done in~\cite{LePa14} for both estimators MLMC and ML2R under the following assumptions on the sequence $(Y_h)_{h \in \Hr}$. The first one, called \emph{bias error expansion} (or \emph{weak error assumption}), states 
\begin{equation} 
	\tag{$WE_{\alpha, \bar{R}}$} \label{lemaireweak_error}
	\exists \,\alpha > 0, \bar R \ge 1, (c_r)_{1\leq r\leq\bar{R}}, \quad \esp{Y_h} - \esp{Y_0} = \sum_{k=1}^{\bar R} c_k h^{\alpha k} + h^{ \alpha\bar R} \eta_{_{\bar R}}(h), \quad \lim_{h\to 0}\eta_{_{\bar R}}(h)=0.
\end{equation}
The second one, called \emph{strong approximation error assumption}, states 
\begin{equation} 
	\tag{$SE_\beta$} \label{lemairestrong_error}
	\exists\, \beta > 0, V_1 \ge 0,  \quad \normLp{2}{Y_h - Y_0}^2 = \esp{\abs{Y_h - Y_0}^2} \le V_1 h^\beta.
\end{equation}
Note that the strong error assumption can be sometimes replaced by the sharper 
\begin{equation*}
	\exists\, \beta > 0, V_1 \ge 0,  \quad \normLp{2}{Y_h - Y_{h'}}^2 = \esp{\abs{Y_h - Y_{h'}}^2} \le V_1 |h-h'|^\beta, \quad h,h'\in\Hr.
\end{equation*}
From now on, we set $I_\pi^N(\varepsilon) := I_{\pi(\varepsilon)}^{N(\varepsilon)}$, where $\pi(\varepsilon)$ and $N(\varepsilon)$ are closed to solutions of \eqref{eq::cost_min} (see~\cite{LePa14} for the construction of these parameters and Tables \ref{tab:opt_param_MLRR} and \ref{tab:opt_param_MLMC} for the explicit values). As mentioned by Duffie and Glynn in~\cite{DuGl95}, the global cost of the standard Monte Carlo with these \emph{optimal} parameters satisfies $$\Cost\pa{I_\pi^N(\varepsilon)} \lesssim K(\alpha) \varepsilon^{-(2+\frac{1}{\alpha})}$$ where the finite real constant $K(\alpha)$ depends on the structural parameters $\alpha, \Var(Y_0), \hh$ and we recall $f(\varepsilon)\lesssim g(\varepsilon)$ iff $\limsup\limits_{\varepsilon\to0} g(\varepsilon)/f(\varepsilon)\leq1$. 
Giles for MLMC in~\cite{Gi08} and Lemaire and Pag\`es for ML2R in~\cite{LePa14} showed that, using these \emph{optimal} parameters the global cost is upper bounded by a function of $\varepsilon$, depending on the weak error expansion rate $\alpha$ and on the strong error rate $\beta$. More precisely, for both estimators we have 
\begin{equation}
\label{eq::cost_upper_bound}
 \Cost\pa{I_\pi^N(\varepsilon)} \lesssim K(\alpha, \beta, M) v(\varepsilon)
\end{equation}
where the finite real constant $K(\alpha, \beta, M)$ is explicit and differs between MLMC and ML2R (see~\cite{LePa14} for more details). Denoting $v_{\MLMC}$ and $v_{\MLRR}$ the dominated function in \eqref{eq::cost_upper_bound} for the MLMC and ML2R estimator respectively, we obtain two distinct cases. In the case $\beta > 1$ both estimators behaves very well as an unbiased Monte Carlo estimator \emph{i.e.} $v_{\MLMC}(\varepsilon) = v_{\MLRR}(\varepsilon) = \varepsilon^{-2}$. In the case $\beta \le 1$, the ML2R is asymptotically quite better than MLMC since $\lim_{\varepsilon \to 0} \frac{v_{\MLRR}}{v_{\MLMC}} = 0$. More precisely, we have 
\begin{table}[!ht] 
	\centering
\begin{tabu} to 130mm {X[1,c]|X[4,c]|X[4,c]}
        & $v_{\MLMC}(\varepsilon)$ & $v_{\MLRR}(\varepsilon)$ \\
		\midrule
        $\beta = 1$ 
        & $\ds \varepsilon^{-2} \log(1/\varepsilon)^{2}$
        & $\ds \varepsilon^{-2} \log(1/\varepsilon)$
        \\ \midrule
        $\beta < 1$ 
        & $\ds \varepsilon^{-2 - \frac{1-\beta}{\alpha}}$
        & $\ds \varepsilon^{-2} e^{\frac{1-\beta}{\sqrt{\alpha}} \sqrt{2\log(1/\varepsilon)\log(M)}}$
        \\
    \end{tabu}
\end{table}

The aim of this paper is to prove a Strong Law of Large Numbers (SLLN) and a Central Limit Theorem (CLT) for both estimators MLMC and ML2R calibrated using these \emph{optimal} parameters. First notice that as these parameters have been computed under the constraint $\normLp{2}{I_\pi^N(\varepsilon) - I_0} \leq \varepsilon$, the convergence in $L^2$  holds by construction. As a consequence, it is straight forward that, for every sequence $(\varepsilon_k)_{k\geq1}$
such that $\sum_{k\geq1}\varepsilon_k^2<+\infty$,
\begin{equation}
	\label{eq::L2_convergence}
	\sum_{k\geq1} \esp{\left| I_\pi^N(\varepsilon_k) - I_0\right|^2} < +\infty,
\end{equation}
so that $$I_\pi^N(\varepsilon_k) \xrightarrow{a.s.} I_0, \quad \text{as $k \to +\infty$}.$$ 
We will weaken the assumption on the sequence $(\varepsilon_k)_{k\geq1}$ when $Y_h$ has higher finite moments, so we will investigate some $L^p$ criterions for $p\geq2$.
Moreover, provided a sharper strong error assumption and adding some more hy\-po\-the\-sis of uniform integrability, we will show that 
$$\frac{I_\pi^N(\varepsilon) - I_0}{\varepsilon} - m(\varepsilon) \convloi \normal{0, \sigma^2}, \quad \text{as $\varepsilon \to 0$},$$
with $m(\varepsilon) = \frac{\mu(\varepsilon)}{\varepsilon}$ where $\mu(\varepsilon) = \esp{I_\pi^N} - I_0$ is the bias of the estimator, and $m^2+\sigma^2\leq1$, owing to the explicit expression of the constraint 
\begin{equation}
 \label{eq::constraint}
 \normLp{2}{I_\pi^N(\varepsilon) - I_0}^2 =  \mu(\varepsilon)^2 + \Var(I_\pi^N(\varepsilon)) \leq \varepsilon^2.
\end{equation}
In particular we will prove that $\lim_{\varepsilon \to 0} m(\varepsilon) = 0$ for the ML2R estimator. 
More precisely  we will use in the proof the expansion
$$\frac{I_\pi^N(\varepsilon) - I_0}{\varepsilon} =   m(\varepsilon) + \sigma_2 \zeta_2^\varepsilon + \frac{1}{\varepsilon \sqrt{N(\varepsilon)}} \sigma_1 \zeta_1^\varepsilon, \quad \mbox{as } \varepsilon\to0,$$
where $\zeta_1^\varepsilon$ and $\zeta_2^\varepsilon$ are two independent variables such that $(\zeta_1^\varepsilon, \zeta_2^\varepsilon) \convloi \mathcal{N}(0,I_2)$ as $\varepsilon\to0$. We will see that $\zeta_1^\varepsilon$ comes from the coarse level of the estimator, while $\zeta_2^\varepsilon$ derives from the sum of the refined levels.  
When $\beta>1$, $\varepsilon \sqrt{N(\varepsilon)}$ converges to a constant, hence the variance $\sigma^2$ results from  the sum of the variance  of the first coarse level $\sigma_1^2$ and the variance of the sum of the refined fine levels $\sigma_2^2$. When $\beta\in(0,1]$, since $\varepsilon \sqrt{N(\varepsilon)}$ diverges, the contribution to $\sigma^2$ of the coarse level disappears and only the variance of the refined levels contributes to $\sigma^2$. 
More details on $m$ and $\sigma$ will follow in Section \ref{section::main_results}. 

The paper is organized as follows. In Section \ref{section::technical_background} we briefly recall the technical background for Multilevel Monte Carlo estimators. In Section \ref{section::main_results} we stable our main results: a Strong Law of Large Numbers and a Central Limit Theorem in a quite general framework. Section \ref{section:auxiliary} is devoted to the analysis of the asymptotic behaviour of the optimal parameters, to the study of the weights of the ML2R estimator and to the bias of the estimators and its robustness. These are auxiliary results that we need for the proof of the main theorems, which we detail in Section \ref{section::proofs}. In Section \ref{section::applications} we apply these results first to the discretization schemes of Brownian diffusions, where we retrieve recent results by Ben Alaya and Kebaier in~\cite{BeKe15}, and secondly to Nested Monte Carlo. 

\

\noindent {\sc Notations:} 

\noindent $\bullet$  Let $\N^* = \{1,2, \ldots\}$ denote the set of positive integers and $\N=\N^*\cup\{0\}$.
	
\noindent $\bullet$  For every $x\!\in \R_+ = [0, +\infty)$, $\lceil x \rceil$ denotes the unique $n \in \N^*$ satisfying $n-1 < x \le n$.% and $\lfloor x \rfloor$ denotes the unique $n \in \N$ satisfying $n \le x <n+1$.
	
\noindent $\bullet$  If $(a_n)_{n \in \N}$ and $(b_n)_{n \in \N}$ are two sequences of real numbers, $a_n \sim b_n$ if $a_n = \varepsilon_n b_n$ with $\lim_n \varepsilon_n = 1$, $a_n = O(b_n)$ if $(\varepsilon_n)_{n \in \N}$ is bounded and $a_n = o(b_n)$ if $\lim_n \varepsilon_n = 0$.

\noindent $\bullet$ $\Var\pa{X}$ and $\sigma(X)$ denote the variance and the standard deviation of a random variable $X$ respectively.

\section{Brief background on MLMC and ML2R estimators}
\label{section::technical_background}
We follow~\cite{LePa14} and recall briefly the construction of the \emph{optimal} parameters derived from the optimization problem \eqref{eq::cost_min}. The first step is a stratification procedure allowing us to establish the optimal allocation policy $(q_1,\dots,q_R)$ when the other parameters $R, h, M$ are fixed. We focus now on the \emph{effort} of the estimator defined as the product of the cost times the variance \emph{i.e.} $\effort(I_\pi^N) = \Cost(I_\pi^N) \Var(I_\pi^N)$. Introducing the notations  
\begin{equation*}
    \forall j \ge 2, \; Z_j = \pa{\frac{h}{M^{j-1}}}^{-\frac{\beta}{2}} \pa{Y^{(j)}_{\frac{h}{M^{j-1}}} - Y^{(j)}_{\frac{h}{M^{j-2}}}} \quad \text{and} \quad Z_1 = Y^{(1)}_{h},
\end{equation*}
a Multilevel estimator MLMC \eqref{intro:MLMC-estimator} or ML2R \eqref{intro:MLRR-estimator} writes  
\begin{equation*}
    I_\pi^N = \frac{1}{N_1} \sum_{k=1}^{N_1} Z_1^k + \sum_{j=2}^R \frac{1}{N_j} \sum_{k=1}^{N_j} W_j^R \pa{\frac{h}{M^{j-1}}}^{\frac{\beta}{2}} Z_j^k
\end{equation*}
where $W_j^R = 1$ for the MLMC and $W_j^R = \W_j^R$ for the ML2R. By definition and using the approximation $N_j \simeq N q_j$ the effort satisfies  
\begin{equation}
    \effort(I_\pi^N) \simeq \pa{\sum_{j=1}^R q_j \Cost(Z_j)} \pa{\frac{\Var(Y_{h})}{q_1} + \sum_{j=2}^R (W_j^R)^2 \pa{\frac{h}{M^{j-1}}}^\beta \frac{\Var(Z_j)}{q_j}}.
\end{equation} 
Given $R, h, M$, a minimization of $q \in (0,1)^R \mapsto \effort(I_\pi^N)$ on $\ds \ac[2]{q \in (0,1)^R, \, \sum_{j=1}^R q_j = 1}$ gives the solution 
\begin{equation} \label{eq:optimal_qj}
    \begin{cases}
        q_1^* = \mu^* \sqrt{\frac{\Var(Y_{h})}{\Cost(Y_{h})}} \\
        q_j^* = \mu^* \pa{\frac{h}{M^{j-1}}}^{\frac{\beta}{2}} \abs{W_j^R} \sqrt{\frac{\Var(Z_j)}{\Cost(Z_j)}}
    \end{cases}
    \quad \text{with $\mu^*$ such that $\sum_{j=1}^R q_j = 1$},
\end{equation}
using the Schwarz's inequality (see Theorem 3.6 in~\cite{LePa14} for a detailed proof). The strong error assumption~\eqref{lemairestrong_error} allows us to upper bound $\Var(Y_{h})$ and $\Var(Z_j)$ by $\Var(Y_0)\pa[1]{1+ \theta h^{\beta/2}}^2$ with $\theta = \sqrt{\frac{V_1}{\Var(Y_0)}}$ and $V_1 \pa[1]{1+M^{\frac{\beta}{2}}}^2$ respectively. On the other hand, we assume that $\Cost(Z_j) = \frac{(1 + M^{-1})}{h} M^{j-1}$. Plugging theses estimates in \eqref{eq:optimal_qj} we obtain the \emph{optimal} allocation policy used in this paper and given in Tables \ref{tab:opt_param_MLRR} and \ref{tab:opt_param_MLMC}. 
Notice that this particular choice for the $q_j$ is not unique, if we change \eqref{lemairestrong_error} with a different strong error assumption, for example with the sharp version, then we have to replace the upper bound for $\Var(Z_j)$ with $V_1|1-M^{\frac\beta2}|^2$ and a new expression for the $q_j$ follows. In the same spirit, the $\Cost(Z_j)$ can be different and hence have an impact on the $q_j$, see~\cite{GeJoCl15} or the $nested$ Monte Carlo methods as examples of alternative costs. 

The second step is to select $h(\varepsilon) \in \Hr$ and $R(\varepsilon) \ge 2$ to minimize the cost of the \emph{optimally allocated} estimator given a prescribed RMSE $\varepsilon > 0$. To do this we use the weak error assumption \eqref{lemaireweak_error} and we obtain 
\begin{equation*}
    h(\varepsilon) = \hh / \left\lceil \hh (1+2 \alpha )^{\frac{1}{2 \alpha}} |c_1|^{\frac 1{\alpha}} \varepsilon^{-\frac{1}{\alpha }} M^{-(R-1)}\right\rceil 
\end{equation*} 
with $c_1$ the first coefficient in the weak error expansion, for the MLMC estimator. For the ML2R estimator we made the additional assumption $\tilde c_{\infty} = \lim_{R \to \infty} \abs{c_R}^{\frac{1}{R}} \in (0, +\infty)$ and then we obtain  
\begin{equation*}
    h(\varepsilon) = \hh / \left\lceil \hh (1+2 \alpha R)^{\frac{1}{2 \alpha R}} \widetilde c_\infty^{\frac 1{\alpha}} \varepsilon^{-\frac{1}{\alpha R}} M^{-\frac{R-1}{2}}\right\rceil.
\end{equation*}
The depth parameter $R \ge 2$ follows and the choice of $N$ is directly related to the constraint \eqref{eq::constraint}.  

We report in Tables \ref{tab:opt_param_MLRR} and \ref{tab:opt_param_MLMC} the  ML2R and  MLMC 
values for $R(\varepsilon)$, $h(\varepsilon)$, $q(\varepsilon)=(q_1(\varepsilon),\ldots, q_R(\varepsilon))$, $N(\varepsilon)$ computed in~\cite{LePa14} and used throughout this paper. Note that these parameters are used in the web application of the LPMA at the address 
\href{http://simulations.lpma-paris.fr/multilevel/}{http://simulations.lpma-paris.fr/multilevel}. The following constants are used in this paper and in the Tables \ref{tab:opt_param_MLRR} and \ref{tab:opt_param_MLMC}  
$$\theta = \sqrt{\frac{V_1}{\Var(Y_0)}} \quad \text{and} \quad \widetilde{c}_\infty = \lim_{R \to \infty} \abs{c_R}^{\frac 1 R} \in(0,+\infty)$$ 
and 
$$\underline{C}_{M,\beta} = \frac{1+M^{\frac \beta 2}}{\sqrt{1+M^{-1}}} \quad \text{and}  \quad \bar{C}_{M,\beta} = \left( 1+M^{\frac\beta 2}\right)\sqrt{1+M^{-1}}.$$
Notice that $1+M^{\frac\beta2}$ comes from the \eqref{lemairestrong_error} and $\sqrt{1+M^{-1}}$ from the cost, hence the constants $\underline{C}_{M,\beta}$ and $\bar{C}_{M,\beta}$ depend on them, but on anything else. 
\begin{table}[!ht] 
	\centering
	\begin{tabu} to 150mm {X[1,c]|X[10,c]}
		$R(\varepsilon)$ & $\ds \left\lceil \frac 12 + \frac{\log\pa[1]{\widetilde c_\infty^{\frac{1}{\alpha}} \mathbf{h}}}{\log(M)}                                                            
		+ \sqrt{ \pa[3]{\frac 12 + \frac{\log\pa[1]{\widetilde c_\infty^{\frac{1}{\alpha}} \mathbf{h}}}{\log(M)} }^2 + 2 \frac{\log\pa{A/\varepsilon}}{\alpha \log(M)}} \;\right\rceil, \quad A = \sqrt{1+4\alpha} $
		\\ 
		\midrule
		$h(\varepsilon)$ & $\hh / \left\lceil \hh (1+2 \alpha R)^{\frac{1}{2 \alpha R}} \widetilde c_\infty^{\frac 1{\alpha}} \varepsilon^{-\frac{1}{\alpha R}} M^{-\frac{R-1}{2}}\right\rceil$ 
		\\
		\midrule
		$q(\varepsilon)$ &                                                                                                                                                                      
		$ 
		\begin{aligned}
		q_1 & = \mu^* (1 +\theta h^{\frac \beta 2})                                                                                                                                \\
		q_j & = \mu^* \theta h^{\frac \beta 2} \underline{C}_{M,\beta} \abs{\W_j(R,M)} M^{-\frac{1+\beta}{2}(j-1)},\; j = 2,\dots,R;\;  \sum_{1\le j\le R}q_j= 1                   \\
		\end{aligned} 
		$
		\\
		\midrule
		$N(\varepsilon)$ & $\ds \pa{1+\frac{1}{2 \alpha R}} \frac{\ds \Var(Y_0)                                                                                                                 
		\pa{1 + \theta h^{\frac{\beta}{2}} + \theta h^{\frac{\beta}{2}} \bar{C}_{M,\beta} \sum_{j=2}^R \abs{\W_j(R,M)} M^{\frac{1-\beta}{2}(j-1)}}} {\ds \varepsilon^2 \mu^* }$ \\
	\end{tabu}
	\caption{Optimal parameters for the ML2R estimator.}
	\label{tab:opt_param_MLRR}
\end{table}

\begin{table}[!ht]
	\centering
	\begin{tabu} to 150mm {X[1,c]|X[10,c]}
		$R(\varepsilon)$ & $\ds \left\lceil 1 + \frac{\log\pa[1]{\abs{c_1}^{\frac{1}{\alpha}} \mathbf{h}}}{\log(M)} + \frac{\log(A / \varepsilon)}{\alpha \log(M)} \right\rceil, \quad A = \sqrt{1+2 \alpha} $ 
		\\ 
		\midrule
		$h(\varepsilon)$ & $\hh / \left\lceil \hh (1+2 \alpha )^{\frac{1}{2 \alpha}} |c_1|^{\frac 1{\alpha}} \varepsilon^{-\frac{1}{\alpha }} M^{-(R-1)}\right\rceil$                                          
		\\
		\midrule
		$q(\varepsilon)$ &                                                                                                                                                                                     
		$ 
		\begin{aligned}
		q_1 & = \mu^* (1 +\theta h^{\frac \beta 2})                                                                                                                                               \\
		q_j & = \mu^* \theta h^{\frac \beta 2} \underline{C}_{M,\beta} M^{-\frac{1+\beta}{2}(j-1)}                                                                                                
		, \; j = 2,\dots,R;  \sum_{1\le j\le R}q_j= 1 \\
		\end{aligned}
		$
		\\
		\midrule
		$N(\varepsilon)$ & $\ds \pa{1+\frac{1 }{2 \alpha}}\frac{\ds \Var(Y_0)                                                                                                                                  
		\pa{1 + \theta h^{\frac{\beta}{2}} + \theta h^{\frac{\beta}{2}} \bar{C}_{M,\beta} \sum_{j=2}^R  M^{\frac{1-\beta}{2}(j-1)}}} {\ds \varepsilon^2 \mu^* }$ \\
	\end{tabu}
	\caption{Optimal parameters for the MLMC estimator.}
	\label{tab:opt_param_MLMC}
\end{table}

In what follows, we will shorter these notations by setting
\begin{equation}
	\label{eq::R_ML2R}
	R(\varepsilon) = \left\lceil C_R^{(1)} + \sqrt{C_R^{(2)} + \frac{2}{\alpha\log(M)} \log\left( \frac 1\varepsilon\right)} \right\rceil
\end{equation}
with $C_R^{(1)} = \frac 12 + \frac{\log\pa[1]{\widetilde c^{\frac{1}{\alpha}} \mathbf{h}}}{\log(M)}$ and
$C_R^{(2)} = \pa[3]{\frac 12 + \frac{\log\pa[1]{\widetilde c^{\frac{1}{\alpha}} \mathbf{h}}}{\log(M)} }^2 + 2 \frac{\log\pa{A}}{\alpha \log(M)} $
for ML2R and
\begin{equation}
	\label{eq::R_MLMC}
	R(\varepsilon) = \left\lceil C_R^{(1)} + \frac{1}{\alpha\log(M)} \log\left( \frac 1\varepsilon\right) \right\rceil
\end{equation}
with $C_R^{(1)} = 1 + \frac{\log\pa[1]{\abs{c_1}^{\frac{1}{\alpha}} \mathbf{h}}}{\log(M)} + \frac{\log(A )}{\alpha \log(M)} $
for MLMC.

\section{Main results}
\label{section::main_results}
The asymptotic behaviour, as $\varepsilon$ goes to $0$, of the parameters given in Tables \ref{tab:opt_param_MLRR} and \ref{tab:opt_param_MLMC} will be exposed in Section~\ref{section:auxiliary}. We proceed here to the analysis of the asymptotic behaviour of the estimator 
$I_\pi^N(\varepsilon) := I_{\pi(\varepsilon)}^{N(\varepsilon)}$ as $\varepsilon\to0$. 

\subsection{Strong Law of Large Numbers}
We will first prove a Strong Law of Large Numbers, namely 

\begin{Theorem}[Strong Law of Large Numbers]
	\label{th::slln}	
	Let $p\geq 2$. Assume \eqref{lemaireweak_error} for all $\bar{R} \geq 1$ and $Y_0\in L^p$. Assume furthermore the following $L^p$--strong error rate assumption
	\begin{equation}
		\label{hp::lemairestrong_error_p}
		\exists\, \beta > 0, V_1^{(p)} \ge 0,  \quad \normLp{p}{Y_h - Y_0}^p = \esp{\abs{Y_h - Y_0}^p} \le V_1^{(p)} h^{\frac \beta 2 p}, \quad h\in\Hr.
	\end{equation}
	Then, for every sequence of positive real numbers $(\varepsilon_k)_{k\geq1}$ such that $\sum_{k\geq 1} \varepsilon_k^p<+\infty$, both MLMC and ML2R estimators satisfy
	\begin{equation}
		\label{eq::slln}
		I_\pi^N(\varepsilon_k) \stackrel{a.s.} {\longrightarrow} I_0, \quad \text{as $k \to +\infty$}.
	\end{equation}
\end{Theorem}

\subsection{Central Limit Theorems}
A necessary condition for a Central Limit Theorem to hold will be that the ratio between the variance of the estimator and $\varepsilon$ converges as $\varepsilon\to0$. It seems intuitive that \eqref{lemairestrong_error} should be reinforced by a sharper estimate as $h \to 0$. We define 
\begin{equation}
	\label{def::tilde_Z_j}
	Z(h):= \left(\frac{h}{M}\right)^{-\frac \beta 2} \left(Y_{\frac{h}{M}} - Y_h \right) \quad \mbox{ and } \quad Z_j := Z \left( \frac{\hh}{n_{j-1}}\right). 
\end{equation}
A necessary condition to obtain a CLT is to assume that $\pa[1]{Z(h)}_{h \in \Hr}$ is $L^2$--uniformly integrable. We state two results, the first one in the case $\beta > 1$ and the second one in the case $\beta \le 1$.

\subsubsection[Case greater than one]{Case $\beta > 1$}
In this case, note that following~\eqref{lemairestrong_error} we have $\sup_{j \ge 1} \Var(Z_j) \le V_1 \pa{1+M^{\frac{\beta}{2}}}^2$.

\begin{Theorem}[Central Limit Theorem, $\beta > 1$] 
	\label{th::clt_g1}
Assume~\eqref{lemairestrong_error} for $\beta > 1$ and that $\pa[1]{Z(h)}_{h \in \Hr}$ is $L^2$--uniformly integrable. We set 
\begin{equation}
	\label{eq::sigma_clt_g1}
	\sigma_1^2 = \frac1{\Sigma} \frac{\Var(Y_{\hh})}{\Var(Y_0)(1 + \theta \hh^{\frac \beta2})} \quad \text{and} \quad 
	\sigma_2^2 = \frac1{\Sigma} \frac{\mathbf{h}^{\frac \beta 2} \sum_{j\geq 2} M^{\frac{1-\beta}{2}(j-1)} \Var(Z_j)  
}{\sqrt{\Var (Y_0) V_1} \underline{C}_{M,\beta} }
\end{equation}
with $\displaystyle \Sigma = \Sigma(M,\beta,\theta, \hh) = \left[ 1 + \theta \hh^{\frac \beta 2} \left( 1 + \bar{C}_{M,\beta} \frac{M^{\frac{1-\beta}{2}}}{1-M^{\frac{1-\beta}{2}}}\right) \right]$.
\begin{itemize}
    \item[$(a)$] ML2R estimator: Assume \eqref{lemaireweak_error} for all $\bar{R} \geq 1$. 
	Then
    \begin{equation}
        \label{eq::clt_ml2r_g1}
        \frac{I^N_\pi(\varepsilon) - I_0}{\varepsilon} \convloi \normal{0, \sigma_1^2 + \sigma_2^2}, \quad \text{as $\varepsilon \to 0$}.
    \end{equation}
    \item[$(b)$] MLMC estimator: Assume \eqref{lemaireweak_error} for $\bar{R} = 1$.
	Then there exists, for every $\varepsilon > 0$, $m(\varepsilon)$ such that $\frac{M^{-\alpha}}{\sqrt{1+2 \alpha}} \le \abs{m(\varepsilon)} \le \frac{1}{\sqrt{1+2\alpha}}$ and 
    \begin{equation}
        \label{eq::clt_mlmc_g1}
        \frac{I^N_\pi(\varepsilon) - I_0}{\varepsilon} - m(\varepsilon) \convloi \normal{0, \frac{2 \alpha}{2 \alpha + 1} \left(\sigma_1^2 +  \sigma_2^2\right)}, \quad \text{as $\varepsilon \to 0$}.
    \end{equation}
\end{itemize}
\end{Theorem}

Note that the variance of the first term $Y_\hh$ associated to the coarse level contributes to the asymptotic variance of the estimator throughout $\sigma_1^2$, while the variances of the correcting levels, $\Var(Z_j), j\geq2$, contribute throughout $\sigma_2^2$.
The ML2R estimator is asymptotically unbiased, whereas the MLMC estimator has an \emph{a priori} non-vanishing bias term. This gain on the bias for ML2R is balanced by the variance, which is reduced of a factor $\frac{2\alpha}{1+2\alpha}$ for MLMC. The constraint \eqref{eq::constraint} yields $\sigma_1^2 + \sigma_2^2 \leq1$, which is easy to verify  if we recall that $\Var(Y_\hh)\leq \Var(Y_0)\left(1+\theta \hh^{\frac \beta2}\right)^2$, $\Var(Z_j) \leq V_1 \left( 1+M^{\frac\beta2}\right)^2$ and $\displaystyle m(\varepsilon)^2\leq \frac{1}{1+2\alpha}$. 

\subsubsection[Case less and equal to one]{Case $\beta \in (0,1]$}

In this case, we make the additional sharper assumption that $\ds \lim_{h \to 0} \normLp{2}{Z(h)}^2 = v_\infty(M, \beta)$. This assumption allows us to identify $\ds \lim_{j \to +\infty} \Var(Z_j)$. More precisely, note that owing to the consistence of the strong and weak error $2 \alpha \ge \beta$ and owing to \eqref{lemaireweak_error} we have
\begin{equation*}
    \esp{Z_j} = \pa{\frac{\hh}{n_j}}^{-\frac{\beta}{2}} \esp{Y_{\frac{\hh}{n_j}} - Y_{\frac{\hh}{n_j}}} = c_1 (1 - M^\alpha) \pa{\frac{\hh}{n_j}}^{\alpha-\frac{\beta}{2}} + o\pa{\pa{\frac{\hh}{n_j}}^{\alpha - \frac{\beta}{2}}},
\end{equation*}
so that  
\begin{equation*}
	\Var(Z_j) = \normLp[3]{2}{Z\pa{\frac{\hh}{n_{j-1}}}}^2 - c_1^2 \left(1 - M^\alpha \right)^2 \left( \frac{\hh}{n_j} \right)^{2\alpha - \beta} + o\left(\left( \frac{\hh}{n_j} \right)^{2\alpha - \beta}\right).
\end{equation*}
We conclude that 
\begin{equation}
    \lim_{j \to +\infty} \Var(Z_j) = 
\begin{cases}
    v_\infty(M, \beta) & \text{if $2 \alpha > \beta$}, \\
    v_\infty(M, \beta) - c_1^2 (1 - M^{\frac{\beta}{2}})^2 & \text{if $2 \alpha = \beta$}.
\end{cases}
\end{equation} 
%Now we are in position to state the Central Limit Theorem in the case $\beta\in(0,1]$.
\begin{Theorem}[Central Limit Theorem, $0 < \beta \le 1$] 
	\label{th::clt_le1}
    Assume~\eqref{lemairestrong_error} for $\beta \in (0, 1]$. Assume that $\pa[1]{Z(h)}_{h \in \Hr}$ is $L^2$--uniformly integrable and assume furthermore $\ds \lim_{h \to 0} \normLp{2}{Z(h)}^2 = v_\infty(M, \beta)$. %assume furthermore \eqref{eq::norm_L_2_convergence}. 
We set 
\begin{equation}
	\label{eq::sigma_clt_le1}
	\sigma^2 = \begin{cases} 
        v_\infty(M, \beta) \left(1+M^{\frac{\beta}{2}}\right)^{-2} V_1^{-1} & \text{if $2\alpha > \beta$}, \\
        \pa{v_\infty(M, \beta) - c_1^2 (1-M^{\frac{\beta}{2}})^2} \left(1+M^{\frac{\beta}{2}}\right)^{-2} V_1^{-1} & \text{if $2\alpha = \beta$.}
    \end{cases}      
\end{equation}
\begin{itemize}
    \item[$(a)$] ML2R estimator: Assume \eqref{lemaireweak_error} for all $\bar{R} \geq 1$. 
    Then
    \begin{equation}
		\label{eq::clt_ml2r_le1}
	    \frac{I^N_\pi(\varepsilon) - I_0}{\varepsilon} \convloi \normal{0,\sigma^2}, \quad \text{as $\varepsilon \to 0$}.
	\end{equation}
    \item[$(b)$] MLMC estimator: Assume \eqref{lemaireweak_error} for $\bar{R} = 1$ and that $2 \alpha > \beta$ when $\beta<1$. Then there exists for every $\varepsilon > 0$, $m(\varepsilon)$ such that $\frac{M^{-\alpha}}{\sqrt{1+2 \alpha}} \le \abs{m(\varepsilon)} \le \frac{1}{\sqrt{1+2\alpha}}$ and 
    \begin{equation}
        \label{eq::clt_mlmc_le1}
		\frac{I^N_\pi(\varepsilon) - I_0}{\varepsilon} - m(\varepsilon) \convloi \normal{0, \frac{2 \alpha}{2 \alpha + 1} \sigma^2}, \quad \text{as $\varepsilon \to 0$}.
    \end{equation}
\end{itemize}
\end{Theorem}

We will see in the proof that the asymptotic variance corresponds to the variance associated to the correcting levels. 

\subsection{Practitioner's corner}
 In the proof of Theorems \ref{th::clt_g1} and \ref{th::clt_le1} we will obtain the more precise expansion
 \begin{equation*}
  \frac{I^N_\pi(\varepsilon) - I_0}{\varepsilon} = m(\varepsilon)  + \Sigma_2 \zeta_2^\varepsilon + \frac{1}{\varepsilon \sqrt{N(\varepsilon)}} \Sigma_1 \zeta_1^\varepsilon, \quad \mbox{as } \varepsilon\to0,
 \end{equation*}
where $\zeta_1^\varepsilon$ and $\zeta_2^\varepsilon$ are two independent variables such that $(\zeta_1^\varepsilon, \zeta_2^\varepsilon) \convloi \mathcal{N}(0, I_2)$ as $\varepsilon\to0$, and the real values $\Sigma_1$ and $\Sigma_2$ depend on whether we are in the MLMC or in the ML2R case and on the value of $\beta$. Fundamentally $\Sigma_1$ comes from the variance of the first coarse level and $\Sigma_2$ from the sum of variances of the correcting levels. 

When $\beta>1$, we will prove in Lemma~\ref{lemma::N_convergence} that $\varepsilon \sqrt{N(\varepsilon)}$ converges to a constant as $\varepsilon\to0$, hence both the coarse and the refined levels contribute to the asymptotic of the estimator. 

When $\beta\leq1$, we will see that $\left(\varepsilon \sqrt{N(\varepsilon)}\right)^{-1}\to0$ as $\varepsilon\to0$ so that, asymptotically, the variance of the coarse level fades and only the refined levels contribute to the asymptotic variance. Still, it is commonly known in the Multilevel framework that the coarse level is the one with the biggest size (speaking in terms of $N_j$), hence this term is not really negligible. We can go through this contradiction by observing the {\em inverse convergence rate to} $0$,  namely $\varepsilon \sqrt{N(\varepsilon)}$. It is equivalent, up to a constant, to $\sqrt{R(\varepsilon)}$ when $\beta=1$ and $M^{\frac{1-\beta}{4} R(\varepsilon)}$ when $\beta<1$. 

- For ML2R, owing to the expression of $R(\varepsilon)$ given in \eqref{eq::R_ML2R},  $\varepsilon \sqrt{N(\varepsilon)} \sim C \big(\log(1/\varepsilon)\big)^{\frac 14}$ where $C$ is a positive constant  when $\beta=1$ and  $\varepsilon \sqrt{N(\varepsilon)} = o\left(\varepsilon^{-\eta}\right)$ for all $\eta>0$ when $\beta<1$. Hence the convergence rate to 0 of $\left(\varepsilon \sqrt{N(\varepsilon)}\right)^{-1}$ is very slow.
By contrast, $\Sigma_1\gg\Sigma_2$, since $\Sigma_1$ is related to the variance of the coarse level which roughly approximates the value of interest whereas $\Sigma_2$ is related to the variance of the refined levels supposed to be smaller a priori. Hence the product $\left(\varepsilon \sqrt{N(\varepsilon)}\right)^{-1}\Sigma_1$ turns out not to be negligible with respect to $\Sigma_2$ for the    values of the RMSE $\varepsilon $   usually prescribed in applications. 

- For MLMC, we get $\varepsilon \sqrt{N(\varepsilon)} \sim C \sqrt{\log(1/\varepsilon)}$, $C$ positive constant, for $\beta=1$ and $\varepsilon \sqrt{N(\varepsilon)} \sim C' \varepsilon^{-\frac{1-\beta}{4\alpha}}$ for $\beta<1$. Hence, when $\beta>1$, the slow convergence phenomenon is still observed though less   significant. 
%and we have to analyze whether $\frac{1-\beta}{4\alpha}<1$ or not. 
%We know that, since $2\alpha\geq\beta$, $4\alpha + \beta\geq 3\beta$, hence $\beta>1/3$ yields $\frac{1-\beta}{4\alpha}<1$.

\

$\rhd$ \emph{Impact of the weights $\W_j^R$, $j=1,\ldots, R$ on the asymptotic behaviour of the ML2R estimator:} 
When $\beta\geq1$, one observes that neither the rate of convergence nor the asymptotic variance of the estimator depends in any way upon the weights $\W_j^R$, $j=1,\ldots, R$. If $\beta<1$ it depends in a somewhat hidden way through the multiplicative constant of $\varepsilon^{-2} M^{\frac{1-\beta}{2}R(\varepsilon)}$ in the asymptotic of $N(\varepsilon)$ (see Lemma \ref{lemma::N_convergence} for more details).  However, at finite range, it may have an impact on the variance  of the estimator, having however in mind that, by construction, the depth of the ML2R estimator is lower than that of the MLMC which tempers this effect.

\section{Auxiliary results}
\label{section:auxiliary}
This Section contains some useful results for the proof of the Strong Law of Large Numbers and of the Central Limit Theorem. More in detail, we investigate the asymptotic behaviour as $\varepsilon\to0$ of the optimal parameters given in Tables \ref{tab:opt_param_MLRR} and \ref{tab:opt_param_MLMC} and of the bias of the estimators and we analyze the weights of the ML2R estimator.
\subsection{Asymptotic of the bias parameter and of the depth}

An important property of MLMC and ML2R estimators is that $h(\varepsilon)\to \hh$ and $R(\varepsilon)\to\infty$ as $\varepsilon\to0$.
The saturation of the bias parameter $\hh$ is not intuitively obvious, indeed it is well known that $h(\varepsilon) \to0$ as $\varepsilon\to0$ for Crude Monte Carlo estimator.
Still, this is a good property, because  $h = \hh$ is the choice which minimizes the cost of simulation of the variable $Y_h$, which we recall is inverse linear with respect to $h$.
First of all, we retrace the computations that led to the choice of the optimal $h^*(\varepsilon)$ and $R^*(\varepsilon)$, starting from ML2R estimator.
We define  
$$h(\varepsilon,R) = (1+2 \alpha R)^{-\frac{1}{2 \alpha R}} |c_R|^{-\frac 1{\alpha R}} \varepsilon^{\frac{1}{\alpha R}} M^{\frac{R-1}{2}}$$
and we recall that this is the optimized bias found in~\cite{LePa14}  at $R$ fixed.
Since the value of $c_R$ is unknown, it is necessary to make the assumption $|c_R|^{\frac1R}\to\widetilde c$ as $R\to+\infty$
and $|c_R|^{-\frac 1{\alpha R}}$ is replaced by $\widetilde c^{-\frac 1 \alpha}$.
The value of $\widetilde c$ is also unknown and in the simulations we have to take an estimate of 
$\widetilde c$, that we write $\hat c$.
We follow the lines of~\cite{LePa14} and define the polynomial
\begin{equation}
	\label{eq::polynomial}
	P(R) = \frac{R(R-1)}{2}\log(M) -R\log(K) - \frac1\alpha \log\left( \frac{\sqrt{1+4\alpha}}{\varepsilon}\right)
\end{equation}
where $K=\hat c^{\frac1\alpha} \hh$. We set $R_+(\varepsilon)$ the positive zero of $P(R)$. 
The optimal value for the depth of the ML2R estimator is $R^*(\varepsilon) = \lceil R_+(\varepsilon)\rceil$.
We notice that $P(R^*(\varepsilon)) \geq 0$, $R^*(\varepsilon)\to +\infty$ as $\varepsilon\to0$, and $R^*$ is 
increasing in $\hat c$.
We can rewrite $h(\varepsilon,R) = \left( \frac{1+4\alpha}{1+2\alpha R}\right)^{\frac{1}{2\alpha R}} \left( \frac{\hat c}{|c_R|^{\frac 1 R}}\right)^{\frac 1 \alpha} e^{\frac{P(R)}{R}} \hh$.
We notice that $h(\varepsilon,R_+) = \left( \frac{1+4\alpha}{1+2\alpha {R_+}}\right)^{\frac{1}{2\alpha {R_+}}} \left( \frac{\hat c}{|c_{R_+}|^{\frac 1 {R_+}}}\right)^{\frac 1 \alpha} \hh$.
The optimal choice for the bias is the projection of $h(\varepsilon,R^*(\varepsilon))$ on the set $\mathcal H = \left\{ \frac{\mathbf{h}}{n}:  n \in \N \right\}$,
which reads $h^*(\varepsilon) =  \hh\left\lceil \frac{\hh}{h(\varepsilon,R^*(\varepsilon))}\right\rceil^{-1}$.
When we replace $|c_R|^{\frac1R}$ with $\hat c$, we finally obtain
$$h^*(\varepsilon) =  \frac{\hh}{\left\lceil \hh (1+2 \alpha R^*)^{\frac{1}{2 \alpha R^*}} \hat c^{\frac 1{\alpha}} \varepsilon^{-\frac{1}{\alpha R^*}} M^{-\frac{R^*-1}{2}}\right\rceil} = \hh \Biggl\lceil \frac{\hh}{h(\varepsilon,R^*) \left(|c_{R^*}|^{\frac 1{R^*}}\hat c^{-1}\right)^{\frac1\alpha}} \Biggr\rceil^{-1}. $$
Let us analyze the denominator $\displaystyle \frac{\hh}{h(\varepsilon,R^*)\left(|c_{R^*}|^{\frac 1{R^*}}\hat c^{-1}\right)^{\frac1\alpha}} =  \left( \frac{1+4\alpha}{1+2\alpha R^*}\right)^{-\frac{1}{2\alpha R^*}} e^{-\frac{P(R^*)}{R^*}}$. 

Since $P(R^*)\geq 0$ and since for $R$ large enough 
the function $\left( \frac{1+4\alpha}{1+2\alpha R}\right)^{-\frac{1}{2\alpha R}} \nearrow 1$, hence, 
up to reducing $\bar\varepsilon$, 
\begin{equation}
	\label{eq::bias_inequality}
	\left( \frac{1+4\alpha}{1+2\alpha R^*}\right)^{-\frac{1}{2\alpha R^*}}  e^{-\frac{P(R^*)}{R^*}}\leq 1, \quad \mbox{for all} \; \varepsilon\in(0,\bar\varepsilon),
\end{equation}
which yields  $\Biggl\lceil \frac{\hh}{h(\varepsilon,R^*)\left(|c_{R^*}|^{\frac 1{R^*}}\hat c^{-1}\right)^{\frac1\alpha}}\Biggr\rceil = 1$ and $h^*(\varepsilon) = \hh$.

For MLMC we may follow the same reasoning starting from $\ds h(\varepsilon,R) = (1+2\alpha)^{-\frac1{2\alpha}} |c_1|^{-\frac1\alpha} \varepsilon^{\frac1\alpha} M^{R-1}$.
We just showed the following
\begin{Proposition}
	\label{rmk::h_bold}
	There exists $\bar{\varepsilon} > 0$ such that $$\forall \varepsilon \in(0, \bar{\varepsilon}] \quad h^*(\varepsilon) = \hh.$$
\end{Proposition}
In what follows, we will always assume that $\varepsilon \in (0, \bar{\varepsilon}]$ and $h^*(\varepsilon)=\hh$.
This threshold $\bar \varepsilon$ can be reduced in what follows line to line.
	
As  $\varepsilon \to 0$, $R=R^*(\varepsilon) \to +\infty$ at the rate 
$\sqrt{\frac{2}{\alpha \log(M)}\log\left(\frac 1 \varepsilon\right)} $ in the ML2R case and 
$\frac{1}{\alpha \log(M)}\log\left(\frac 1 \varepsilon\right)$ in the MLMC case.

\subsection{Asymptotic of the bias and robustness}
\label{section::bias_analysis}
As part of a Central Limit Theorem, we will be faced to the quantity $\displaystyle \frac{\mu(\hh, R(\varepsilon), M)}{\varepsilon}$, where $\mu(\hh, R(\varepsilon), M) = \esp{I_\pi^N(\varepsilon)} - I_0$ is the bias of the estimator.
This leads to analyze carefully its asymptotic behavior as $\varepsilon \to 0$.
Under the \eqref{lemaireweak_error} assumption, the bias of a Crude Monte Carlo estimator reads 

$$\mu(h) = c_1 h^\alpha (1+\eta_1(h)), \quad \lim_{h\to0} \eta_1(h) = 0. $$
The bias of Multilevel estimators is dramatically reduced 
compared to the Crude Monte Carlo, more precisely the following Proposition is proved in~\cite{LePa14}:

\begin{Proposition} 
	\label{prop::bias}
	\hspace{2em}
	\begin{itemize}
		\item[(a)] MLMC: Assume \eqref{lemaireweak_error} with $\bar{R} = 1$.
		      \begin{equation}
		      	\label{eq::bias_MLMC}
		      	\mu(h, R, M) = c_1 \left( \frac{h}{M^{R-1}}\right)^\alpha\left(1+\eta_1\left(\frac{h}{M^{R-1}}\right)\right) 
		      \end{equation}
		      with $\lim_{h\to0}\eta_1(h) = 0$. 
		      
		\item[(b)] ML2R: Assume \eqref{lemaireweak_error} for all $\bar{R} \geq1$.
		      \begin{equation}
		      	\label{eq::bias_ML2R}
		      	\mu(h, R, M) = (-1)^{R-1}c_R \left( \frac{h^R}{M^{\frac{R(R-1)}{2}}}\right)^\alpha(1+\eta_{R, \underline{n}}(h))
		      \end{equation}
		      where $\eta_{R, \underline{n}}(h) = (-1)^{R-1}M^{\alpha\frac{R(R-1)}{2}} \sum_{r=1}^R \frac{w_r}{n_r^{\alpha R}}\eta_R(\frac{h}{n_R})$ with $\lim_{h\to0}\eta_R(h) = 0$.
	\end{itemize}
\end{Proposition}
We notice that the ML2R estimator requires and takes full advantage of a higher order of the expansion of the bias error 
\eqref{lemaireweak_error}, 
whereas the MLMC estimator only needs a first order expansion.
As the computations were made under the constraint $\normLp{2}{I_\pi^N - I_0} \leq \varepsilon$, 
we have clearly that $\frac{|\mu(\hh, R(\varepsilon), M)|}{\varepsilon}\leq 1$. 
We focus our attention on the constants $\widetilde c_\infty$ and $c_1$, which a priori we do not know and that we replace in the simulations 
by $\hat c_\infty = \hat c_1 = 1$.
If we plug the values of $h(\varepsilon) = \hh$ and $R(\varepsilon)$ in the
formulas for the bias, owing to \eqref{eq::polynomial} and \eqref{eq::bias_inequality}  we get, for ML2R,
\begin{align*}
	|\mu(\hh,R(\varepsilon),M)| & = 	|c_{R(\varepsilon)}| \left( \frac{\hh^{R(\varepsilon)}}{M^{R(\varepsilon)\frac{R(\varepsilon)-1}{2}}}\right)^\alpha  =  |c_{R(\varepsilon)}| \hh^{\alpha R(\varepsilon)} \frac{1}{e^{\alpha P(R(\varepsilon))} K^{\alpha R(\varepsilon)}} \frac{\varepsilon}{\sqrt{1+4\alpha}} \\
	& =  \frac{|c_{R(\varepsilon)}|}{\hat c_\infty ^{R(\varepsilon)}} \frac{1}{e^{\alpha P(R(\varepsilon))}} \frac{\varepsilon}{\sqrt{1+4\alpha}}  \leq \frac{|c_{R(\varepsilon)}|}{\hat c_\infty ^{R(\varepsilon)}} \frac{1}{\sqrt{1+2\alpha R(\varepsilon)}} \;\varepsilon                                 
\end{align*}
and, for MLMC,
\begin{equation}
	\left| \frac{c_1}{\hat c_1}\right| \frac{M^{-\alpha}}{\sqrt{1+2\alpha}}\varepsilon < |\mu(\hh,R(\varepsilon),M)| \leq \left| \frac{c_1}{\hat c_1}\right| \frac{1}{\sqrt{1+2\alpha}}\varepsilon.
\end{equation}
We set $m(\varepsilon) := \frac{|\mu(\hh,R(\varepsilon),M)|}{\varepsilon}$. Hence, when taking the true values $\hat c_\infty = \widetilde c_\infty$ and $\hat c_1 = c_1$, we get
\begin{equation}
	\label{eq::bias_on_epsilon_lim}
	\begin{cases}
		\displaystyle \lim_{\varepsilon\to0} m(\varepsilon) = 0 						& \mbox{ for } ML2R, \\
		\frac{M^{-\alpha}}{\sqrt{1+2\alpha}} < \displaystyle  \lim_{\varepsilon\to0} m(\varepsilon) \leq \frac{1}{\sqrt{1+2\alpha}} 	& \mbox{ for } MLMC. 
	\end{cases}
\end{equation}
For ML2R estimators, if $c_R$ has a polynomial growth depending on $R$ $|c_R|^{\frac1R}\to1$ 
and $\hat c_\infty = 1$ corresponds to the exact value of $\widetilde c_\infty$. If the growth of $c_R$  is less than polynomial, 
the convergence to 0 in \eqref{eq::bias_on_epsilon_lim} still holds. 
The only uncertain case is when the growth of $c_R$ is faster than polynomial. 
Then, if $\hat c_\infty \geq |c_R|^{\frac1R}$, $|\mu(\varepsilon)|/\varepsilon$ goes to 0 faster than $\frac{1}{\sqrt{1+2\alpha R(\varepsilon)}}$, 
but if we had taken $\hat c_\infty<1$, we would have obtained $\lim_{R\to+\infty} \frac{|c_{R(\varepsilon)}|}{\hat c_\infty ^{R(\varepsilon)}} =+\infty$,
hence $\hat c_\infty<1$ is definitely not a good choice.
In conclusion, whenever the growth of $c_R$ is at most polynomial, $\hat c_\infty = 1$ remains a good choice. 
When the growth is faster than polynomial it is better to overestimate $\hat c_\infty$ than to underestimate it.
The remarkable fact is that, when we choose $\hat c_\infty$, we are not forced to have a very precise idea of the expression of $c_R$, 
but only of its growth rate.
The choice of $\hat c_1$ for MLMC estimator is less robust, since it is obvious that if we overestimate $c_1$ the inequality
$|\mu(\varepsilon)|/\varepsilon\leq 1/\sqrt{1+2\alpha}$ still holds, but if we underestimate it we eventually
may not have $\frac{|\mu(\varepsilon)|}{\varepsilon}\leq 1$ as expected. Hence the bias for the MLMC estimator 
is very connected to an accurate enough estimation of $c_1$. 

In Figures \ref{fig::c1_estimated_eu} and \ref{fig::c1_estimated_mil} we show the values of $|c_1|$ estimated with the formula
$$c_1 = \left( \frac h2-h \right)^{-1}\pa{\esp{Y_{\frac h2}}-\esp{Y_h}}$$
compared to the value plugged in the simulations $\hat c_1 = 1$, for a Call option in a Black-Scholes model with $X_0=100$, $K=80$, $T=1$, $\sigma=0.4$ and 
making the interest rate vary as follows $r=0.01,0.1,0.2,\ldots, 0.9,1$. 
We simulated $\esp{Y_h}$, with $h=T/20$, using an Euler and a Milstein discretization scheme
and making a Crude Monte Carlo simulation of size $N=10^8$.

In Figures \ref{fig::empirical_bias_eu} and \ref{fig::empirical_bias_mil}
we show the absolute value of the empirical bias for different values of $r$. 
In the simulations, we fixed $\hat c_1 = 1$ and $\hat c_\infty = 1$.
We can observe that when $|c_1|$ is underestimated, the bias for MLMC and 
Crude Monte Carlo estimators do not satisfy the constraint $|\mu(\varepsilon)|\leq \varepsilon$, whereas
the ML2R estimator appears to be less sensible to the estimation of $\tilde c$.

\begin{figure}[htp]
  \centering
  \subfloat[Euler scheme  ($\alpha=1,\beta=1$)]{\label{fig::c1_estimated_eu}\includegraphics[width=0.5\linewidth]{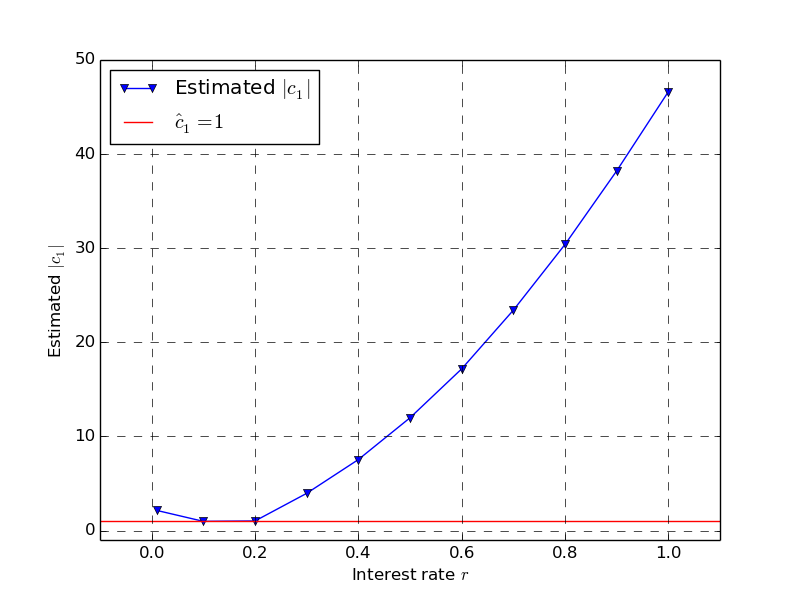} }
  \subfloat[Milstein scheme  ($\alpha=1,\beta=2$).]{\label{fig::c1_estimated_mil}\includegraphics[width=0.5\linewidth]{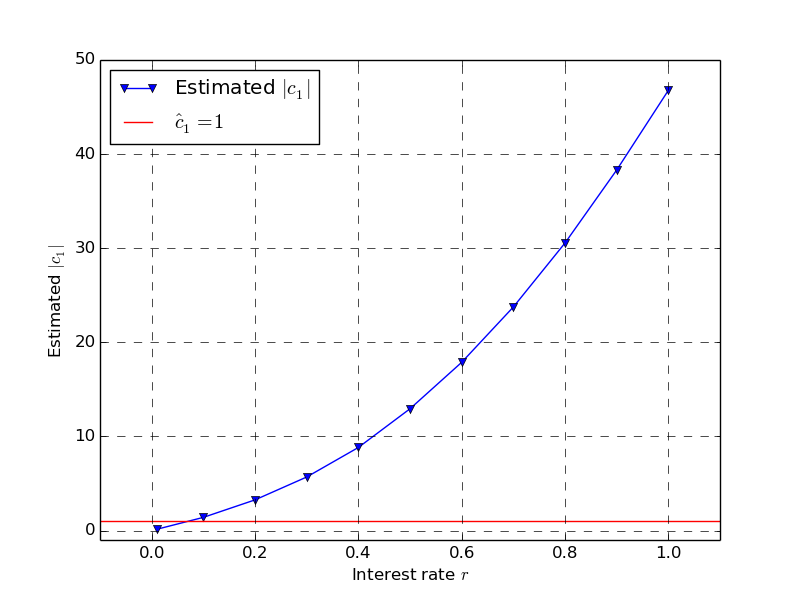}}
  \caption{Estimated $\displaystyle |c_1| = \frac{\left|\esp{Y_h}-\esp{Y_{\frac h2}}\right|}{h-\frac h2}$ when $r$ varies.}
  \subfloat[Euler scheme  ($\alpha=1,\beta=1$)]{\label{fig::empirical_bias_eu}\includegraphics[width=0.5\linewidth]{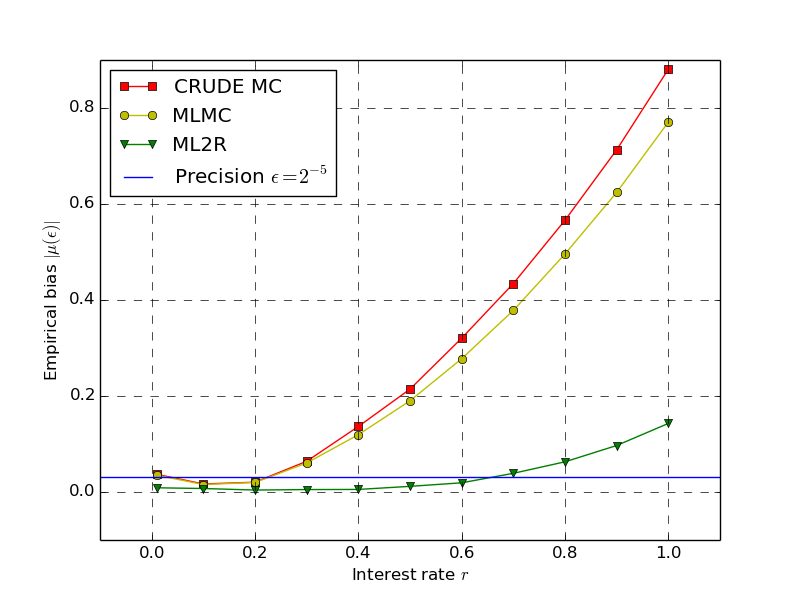}}
  \subfloat[Milstein scheme  ($\alpha=1,\beta=2$).]{\label{fig::empirical_bias_mil}\includegraphics[width=0.5\linewidth]{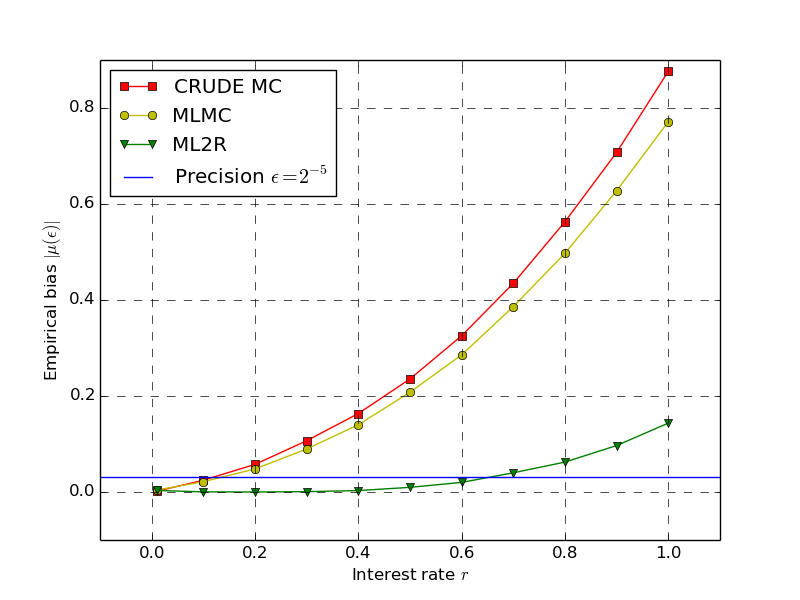}}
  \caption{Empirical bias $|\mu(\varepsilon)|$ for a Call option in a Black-Scholes model for a prescribed RMSE
		$\varepsilon=2^{-5}$ and for different values of $r$, taking $\hat c_\infty = \hat c_1 = 1$.}
\end{figure}

\subsection{Properties of the weights of the ML2R estimator}
\label{section::weights}

One significant difficulty in the proof of the Central Limit Theorem that we stated in Theorems \ref{th::clt_g1} and \ref{th::clt_le1}, 
is to deal with the weights $\W_j^R$ appearing in the ML2R estimator.
Moreover, the analysis of the behaviour of the weights is necessary when studying the asymptotic of the parameters 
$q=(q_1,\ldots, q_R)$ and $N$.
These weights are devised to kill the coefficients $c_1,\ldots,c_R$ in the bias expansion under the \eqref{lemaireweak_error}. They are defined as 
\begin{equation}
	\label{eq::weights}
	\W_j^R = \sum_{r=j}^R \mathbf{w}_r, \quad j=1,\ldots, R,
\end{equation} 
where the weights $\mathbf{w} = (\mathbf{w}_r)_{r=1, \ldots, R}$ 
are the solution to the Vandermonde system $V \w = e_1$, the matrix $V$ being defined by
\begin{equation}
	\label{eq::vandermonde_system}
	V = V(1,n_2^{-\alpha},\dots,n_{_R}^{-\alpha}) = \left( 
	\begin{array}{cccc}
		1      & 1                  & \cdots & 1                     \\
		1      & n_2^{-\alpha}      & \cdots & n_{_R}^{-\alpha}      \\
		\vdots & \vdots             & \cdots & \vdots                \\
		1      & n_2^{-\alpha(R-1)} & \cdots & n_{_R}^{-\alpha(R-1)} \\
	\end{array} \right).
\end{equation}
Notice that $\W_1^R = 1$ by construction. 
In order to give a more tractable expression of the weights $\W_j^R$, one notices that 
the weights $\mathbf{w}$ admit a closed form given by Cramer's rule, namely
$$\mathbf{w}_\ell = a_\ell b_{R-\ell}, \quad \ell=1,\ldots, R,$$
where $a_\ell = \frac{1}{\prod_{1\leq k \leq \ell-1} (1- M^{-k\alpha}) }$, $\ell=1,\ldots, R$, 
with the convention $\prod_{k=1}^0 (1- M^{-k\alpha}) = 1$, and 
$b_{\ell} = (-1)^{\ell} \frac{M^{-\frac\alpha 2 \ell(\ell+1)}}{\prod_{1\leq k \leq \ell} (1- M^{-k\alpha})}, \quad \ell=0,\ldots, R$.
As a consequence

$$\W_j^R = \sum_{\ell=j}^R a_\ell b_{R-\ell} = \sum_{\ell=0}^{R-j} a_{R-\ell} b_{\ell}, \quad j\in{1,\ldots, R}.$$
We will make an extensive use of the following properties, which are proved in Appendix \ref{appendix::weights}.

\begin{Lemma}
	\label{lemma::W_j_R_limits}
	Let $\alpha>0$ and the associated weights $(\W_j^R)_{j=1,\ldots,R}$ given in \eqref{eq::weights}.
	%\hspace{2em}
	\begin{itemize}
		\item[(a)] $\displaystyle \lim_{\ell\to+\infty} a_\ell = a_\infty<+\infty$ and $\displaystyle \sum_{\ell=0}^{+\infty} |b_\ell| = \widetilde{B}_\infty < +\infty$.
		\item[(b)] The weights $\W_j^R$ are uniformly bounded,
		      \begin{equation}
		      	\label{eq::sup_W}
		      	\forall R \in \N^*, \forall j\in \left\{ 1, \ldots, R\right\}, \qquad \left| \W_j^R\right| \leq a_\infty \widetilde{B}_\infty.
		      \end{equation} 
		\item[(c)] For every $\gamma>0$,
		      \begin{equation*}
		      \lim_{R\to+\infty} \sum_{j=2}^R \left| \W_j^R \right| M^{-\gamma(j-1)} = \frac{1}{M^{\gamma}-1} .	
		      \end{equation*}
		\item[(d)] Let $\left\{ v_j\right\}_{j\geq1}$ be a bounded sequence of positive real numbers.
		      Let $\gamma\in\R$ and assume that $\ds \lim_{j\to+\infty} v_j = 1$ when $\gamma\geq 0$.
		      Then the following limits hold:
		      \begin{equation*}
		      	\sum_{j=2}^{R} \left| \W_j^{R}\right| M^{\gamma(j-1)} v_j \stackrel{R\to +\infty}{\sim}
		      	\begin{cases}
		      		\sum_{j\geq 2} M^{\gamma(j-1)} v_j < + \infty                                                & \text{for} \quad \gamma<0,  \\
		      		R                                                                                            & \text{for} \quad \gamma=0,  \\
		      		M^{\gamma R} a_\infty \sum_{j\geq 1} \left| \sum_{\ell=0}^{j-1} b_\ell \right| M^{-\gamma j} & \text{for} \quad \gamma>0 . 
		      	\end{cases}
		      \end{equation*}
	\end{itemize} 
\end{Lemma}

\subsection{Asymptotic of the allocation policy and of the size}
Let us analyze the allocation policy $q = (q_1,\ldots, q_{R})$ for the ML2R case. Since
\begin{equation}
	\label{eq::q_j}
	q_1(\varepsilon) = \mu^*(\varepsilon) \left(1+\theta \hh^{\frac\beta2}\right) \quad \mbox{and} \quad q_j(\varepsilon) = \theta \hh^{\frac \beta 2}  \underline{C}_{M,\beta}  \mu^*(\varepsilon)  \left|\W_j^{R(\varepsilon)}\right| M^{-\frac{\beta+1}{2}(j-1)}, \quad j=2,\ldots, R(\varepsilon),
\end{equation}
the condition $\sum_{j=1}^{R(\varepsilon)} q_j = 1$ yields

$$\mu^*(\varepsilon) = \left( 1 + \theta \hh^{\frac \beta 2} \left( 1 + \underline{C}_{M,\beta} \sum_{j=2}^{R(\varepsilon)} \left|\W_j^{R(\varepsilon)}\right| M^{-\frac{\beta +1}{2}(j-1)} \right) \right)^{-1}.$$
Owing to Lemma \ref{lemma::W_j_R_limits} $(c)$ with $\gamma = \frac{\beta+1}{2}$, the limit of this term as $\varepsilon \to 0$ is

$$\boldsymbol{\mu}^* = \left( 1 + \theta \mathbf{h}^{\frac \beta 2} \left( 1 + \frac{\underline{C}_{M,\beta}}{M^{\frac{1+\beta}{2}}-1} \right)\right)^{-1}.$$
Moreover, for all $\varepsilon\in(0,\bar{\varepsilon}]$, the following inequalities hold: 

\begin{equation}
	\label{rmk::mu_bounds}
	\frac{1}{\bar{\mu}^*}:= 1 + \theta \hh^{\frac \beta 2} \leq \frac{1}{\mu^*(\varepsilon)} \leq  1 + \theta \hh^{\frac \beta 2} \left( 1 + \underline{C}_{M,\beta} a_\infty \widetilde{B}_\infty \frac{1}{1- M^{-\frac{\beta+1}{2}}}\right) =: \frac{1}{\underline{\mu}^*}. 
\end{equation}

\begin{Remark}
	If we set $\W_j^{R(\varepsilon)} = 1$ for all $j=1, \ldots, R(\varepsilon)$, and $a_\infty \widetilde B_\infty = 1$, we obtain the same results for the MLMC allocation policy.
\end{Remark}
The asymptotic of the estimator size $N = N(\varepsilon)$ is given in the following Lemma.

\begin{Lemma}
	\label{lemma::N_convergence}
	$N=N(\varepsilon) \to +\infty$, as $\varepsilon \to 0$, with a convergence rate depending on $\beta$ as follows:
	
	\noindent \textbf{Case $\beta>1$}:
	
	\begin{equation*}
		N(\varepsilon) \sim C_\beta \varepsilon^{-2}
	\end{equation*}
	with 
	\begin{equation}
	\label{eq::const_C_beta_ge_1}
	 C_\beta = \displaystyle \frac{\Var (Y_0)}{\boldsymbol{\mu}^*} \left[ 1 + \theta \hh^{\frac \beta 2} \left( 1 + \bar{C}_{M,\beta} \frac{M^{\frac{1-\beta}{2}}}{1-M^{\frac{1-\beta}{2}}}\right) \right]
		\begin{cases}
			1 & \mbox{for ML2R,}  \\
			\displaystyle 1 + \frac1{2\alpha} & \text{for MLMC.} 
		\end{cases}
	\end{equation}

	\noindent \textbf{Case $\beta\leq1$}: We recall the expression of $R(\varepsilon)$ given in \eqref{eq::R_ML2R} for ML2R and \eqref{eq::R_MLMC} for MLMC.
	
	\begin{equation*}
		N(\varepsilon) \sim C_\beta \varepsilon^{-2}
		\begin{cases}
			R(\varepsilon)                      & \text{if} \quad \beta=1, \\
			M^{\frac{1-\beta}{2}R(\varepsilon)} & \text{if} \quad \beta<1 
		\end{cases}
	\end{equation*}
	where the constant $C_\beta$ reads

	\begin{equation}
		\label{eq::const_C_beta_eq_1}
		\displaystyle
		C_\beta = \frac{\Var (Y_0)}{\boldsymbol{\mu}^*} \theta \hh^{\frac 1 2} \bar{C}_{M,\beta}
		\begin{cases}
			1 					& \mbox{for ML2R,}  \\
			\displaystyle 1 + \frac1{2\alpha}  	& \text{for MLMC}  
		\end{cases}
		\quad \mbox{if} \quad \beta=1 
	\end{equation}
	and
	\begin{equation}
		\label{eq::const_C_beta_le_1}
		\displaystyle
		C_\beta = \frac{\Var (Y_0)}{\boldsymbol{\mu}^*} \theta \hh^{\frac \beta 2} \bar{C}_{M,\beta}
		\begin{cases}
			\displaystyle a_\infty \sum_{j\geq1} \left| \sum_{\ell = 0}^{j-1} b_\ell \right| M^{\frac{\beta-1}{2}j} & \mbox{for ML2R,}  \\
			\displaystyle \left(1 + \frac1{2\alpha}\right) \frac{1}{M^{\frac{1-\beta}{2}}-1} & \text{for MLMC}   
		\end{cases}
		\quad \mbox{if} \quad \beta<1. 
	\end{equation}
	
\end{Lemma}
We notice that for $\beta \geq 1$  the asymptotic behaviour of $N(\varepsilon)$ for ML2R does not depend on the  weights $\W_j^R$ and the difference between the coefficient $C_\beta$ for ML2R and for MLMC estimator lies only in the factor $\left(1+\frac1{2\alpha}\right)$, whereas when $\beta<1$  the asymptotic of the weights has an impact on the behaviour of $N(\varepsilon)$ for ML2R. Still, in this case we observe that if $a_\infty = 1$ and $\left| \sum_{\ell = 0}^{j-1} b_\ell \right|=1$ for all $j\geq1$, hence $a_\infty \sum_{j\geq1} \left| \sum_{\ell = 0}^{j-1} b_\ell \right| M^{\frac{\beta-1}{2}j} = \frac{1}{M^{\frac{1-\beta}{2}}-1}$ and the factor $\left(1+\frac1{2\alpha}\right)$ appears again to be the only difference in the coefficient $C_\beta$ of $N(\varepsilon)$ for the two estimators.

\begin{proof}
	$\rhd$ ML2R: $N$ reads
	
	\begin{equation*}
		N = N(\varepsilon) = \left( 1  + \frac 1 {2 \alpha R(\varepsilon)}\right) \frac{\Var (Y_0)}{\mu^*} \frac{1}{\varepsilon^2} \left( 1 + \theta \hh^{\frac \beta 2} + \theta \hh^{\frac \beta 2} \bar{C}_{M,\beta} \sum_{j=2}^{R(\varepsilon)} \left| \W_j^{R(\varepsilon)} \right| M^{\frac {1-\beta}{2}(j-1)}  \right). 
	\end{equation*}
	We notice that $R(\varepsilon)\to+\infty$ as $\varepsilon\to0$ and use Lemma \ref{lemma::W_j_R_limits} $(d)$ with $\gamma = \frac{1-\beta}{2}$, with $v_j=1$ for each $j\geq1$, to complete the proof on the ML2R framework.
	 
	\
	
	$\rhd$ MLMC: The result follows directly from the convergence of the series $\sum_{j=2}^{R(\varepsilon)} M^{\frac {1-\beta}{2}(j-1)}$, since $N$ reads
	 
	$$N = N(\varepsilon) = \left( 1  + \frac 1 {2 \alpha}\right) \frac{\Var (Y_0)}{\mu^*} \frac{1}{\varepsilon^2} \left( 1 + \theta \hh^{\frac \beta 2} \left(1 + \bar{C}_{M,\beta} \sum_{j=2}^{R(\varepsilon)} M^{\frac {1-\beta}{2}(j-1)}  \right)\right).$$
\end{proof}

\section{Proofs}
\label{section::proofs}
We will use the notations

$$\widetilde{I}_\varepsilon^1:= \frac{1}{N_1(\varepsilon)} \sum_{k=1}^{N_1(\varepsilon)} Y_{\hh}^{(1),k} - \esp{Y_{\hh}^{(1),k}}
\quad \mbox{ and } \quad \widetilde{I}_\varepsilon^2:= \sum_{j=2}^{R(\varepsilon)} \frac{\W_j^{R(\varepsilon)}}{N_j(\varepsilon)} \sum_{k=1}^{N_j(\varepsilon)} \widetilde{Y}_j^k$$
where we set

\begin{equation}
	\label{eq::Y_widetilde}
	\widetilde{Y_j}:= Y_{\frac{\hh}{n_j}}^{(j)} - Y_{\frac{\hh}{n_{j-1}}}^{(j)} - \esp{Y_{\frac{\hh}{n_j}}^{(j)} - Y_{\frac{\hh}{n_{j-1}}}^{(j)}}, \quad j=1,\ldots,R(\varepsilon).
\end{equation}
These notations hold for both ML2R and MLMC estimators, where we set $\W_j^{R(\varepsilon)} = 1$, $j=1,\ldots, R(\varepsilon)$, for MLMC estimators.  
We notice that

\begin{equation}
	\label{eq::I_1_I_2}
	I_{\pi}^{N}(\varepsilon) - I_0 = \widetilde{I}^1_\varepsilon + \widetilde{I}^2_\varepsilon +\mu(\hh,R(\varepsilon),M)
\end{equation}
where the bias $\mu(\hh,R(\varepsilon),M) \to 0$ as $\varepsilon\to 0$ (see Section \ref{section::bias_analysis} for a detailed description of the bias).

\subsection{Proof of Strong Law of Large Numbers}
\label{section::SLLN}

The proof of the Strong Law of Large Numbers is a consequence of the following Proposition.

\begin{Proposition}
\label{prop::L_p_convergence}
 Let $p\geq2$. There exists a positive real constant $K(M,\beta,p)$ such that 
 \begin{equation}
 \label{eq::L_p_convergence}
	\esp{\left| \widetilde{I}_\varepsilon^2 \right|^p} \leq K(M,\beta,p) \varepsilon^p.
 \end{equation}
\end{Proposition}

\begin{proof}
  $\rhd$ ML2R: We first give the proof of \eqref{eq::L_p_convergence} for the ML2R estimator.
  As a first step we show that, for all $p\geq2$, 
  
  \begin{equation}
	  \label{eq::y_j_p}
	  \esp{ \left| \widetilde{Y}_j \right|^p} \leq C_{M,\beta,p} M^{-\frac{\beta p}{2}(j-1)}, \quad j=1,\ldots,R(\varepsilon), \quad \mbox{with} \; C_{M,\beta,p} = 2^p V_1^{(p)} \left(1 + M^{\frac \beta2}\right)^p \hh^{\frac{\beta p}{2}}.
  \end{equation}
  By Minkowski's Inequality 
  
  \begin{multline*}
	  \left( \esp{\left| \widetilde{Y}_j \right|^{p}} \right)^{\frac 1 {p}} \leq \left\| Y_{\frac{\mathbf{h}}{n_j}} - Y_{\frac{\mathbf{h}}{n_{j-1}}} \right\|_{p} + \left| \esp{Y_{\frac{\mathbf{h}}{n_j}} - Y_{\frac{\mathbf{h}}{n_{j-1}}}} \right| \\
	  \leq \left\| Y_{\frac{\mathbf{h}}{n_j}} - Y_{\frac{\mathbf{h}}{n_{j-1}}} \right\|_{p} + \left\| Y_{\frac{\mathbf{h}}{n_j}} - Y_{\frac{\mathbf{h}}{n_{j-1}}} \right\|_{1} \leq 2  \left\| Y_{\frac{\mathbf{h}}{n_j}} - Y_{\frac{\mathbf{h}}{n_{j-1}}} \right\|_{p}.
  \end{multline*}
  Applying again Minkowski's Inequality, the $L^p$-strong approximation assumption \eqref{hp::lemairestrong_error_p} yields \eqref{eq::y_j_p}.
  
  As the random variables $\left( \widetilde{Y}_j^k\right)_{k\geq 1}$ are \iid and the $(\widetilde{Y}_j)_{j=1,\ldots,R(\varepsilon)}$ are centered and independent,
  Rosenthal's Inequality (see~\cite{HaHe80}, Theorem 2.12, p. 23) and 
\eqref{eq::y_j_p} imply
  
  \begin{align*}
	  \esp{\left| \widetilde{I}_\varepsilon^2 \right|^p} & = \esp{\left| \sum_{j=2}^{R(\varepsilon)} \sum_{k=1}^{N_j(\varepsilon)} \frac{\W_j^{R(\varepsilon)}}{N_j(\varepsilon)} \widetilde{Y}_j^k \right|^p}                                                                                                                                                                                                           \\
							    & \leq C_p\left[ \left( \sum_{j=2}^{R(\varepsilon)} \sum_{k=1}^{N_j(\varepsilon)} \frac{|\W_j^{R(\varepsilon)}|^2}{N_j(\varepsilon)^2} \esp{\left(\widetilde{Y}_j^k\right)^2} \right)^{\frac p2} + \sum_{j=2}^{R(\varepsilon)} \sum_{k=1}^{N_j(\varepsilon)} \esp{ \left| \frac{\W_j^{R(\varepsilon)}}{N_j(\varepsilon)}  \widetilde{Y}_j^k \right|^p}  \right] \\
							    & = C_p\left[ \left( \sum_{j=2}^{R(\varepsilon)} \frac{|\W_j^{R(\varepsilon)}|^2}{N_j(\varepsilon)} \esp{\left(\widetilde{Y}_j^k\right)^2} \right)^{\frac p2} + \sum_{j=2}^{R(\varepsilon)} \frac{|\W_j^{R(\varepsilon)}|^p}{N_j(\varepsilon)^{p-1}} \esp{\left| \widetilde{Y}_j \right|^p}  \right]                                                            \\
							    & \leq C_p\left[ C_{M,\beta,2}^{\frac p2} \left( \sum_{j=2}^{R(\varepsilon)} \frac{|\W_j^{R(\varepsilon)}|^2}{N_j(\varepsilon)} M^{-\beta(j-1)} \right)^{\frac p2} + C_{M,\beta,p}^p \sum_{j=2}^{R(\varepsilon)} \frac{|\W_j^{R(\varepsilon)}|^p}{N_j(\varepsilon)^{p-1}}  M^{-\frac{\beta p}{2}(j-1)} \right]                                                 
  \end{align*}
  where $C_p$ is a positive universal real constant.
  As $N_j(\varepsilon) = \lceil N(\varepsilon) q_j(\varepsilon)\rceil \geq N(\varepsilon) q_j(\varepsilon)$, we derive that 
  $$\frac1{N_j(\varepsilon)}\leq \frac{1}{N(\varepsilon)q_j(\varepsilon)},  \quad j=1,\ldots, R(\varepsilon).$$
  It follows from the expression of $q_j$ given in \eqref{eq::q_j} and from inequality \eqref{rmk::mu_bounds} that 
  
  $$\frac{|\W_j^{R(\varepsilon)}|}{q_j(\varepsilon)} \leq \frac{1}{\theta \hh^{\frac \beta 2}  \underline{C}_{M,\beta}  \underline{\mu}^*} M^{\frac{\beta+1}{2}(j-1)}, \quad j=2, \ldots, R(\varepsilon).$$
  Then, using that $\sup\limits_{\substack{j\in\{1, \ldots, R\}, R\geq1}} |\W_j^R|\leq a_\infty \widetilde{B}_\infty $, we get
  $$
  \esp{\left| \widetilde{I}_\varepsilon^2 \right|^p} \leq C_p\left[ A_1 \left( \frac{1}{N(\varepsilon)} \sum_{j=2}^{R(\varepsilon)} M^{\frac{1-\beta}{2}(j-1)} \right)^{\frac p2} + A_2 \frac{1}{N(\varepsilon)^{p-1}} \sum_{j=2}^{R(\varepsilon)}  M^{\frac{p - (\beta+1)}{2}(j-1)} \right]
  $$
  with $A_1 = A_1(M,\beta,p) = C_{M,\beta,2}^{\frac p2} \left( a_\infty \widetilde{B}_\infty \right)^{\frac p2} \left(\theta \hh^{\frac \beta 2}  \underline{C}_{M,\beta}  \underline{\mu}^* \right)^{-\frac p2}$ 
  and \\$A_2 = A_2(M,\beta,p) = C_{M,\beta,p}^p a_\infty \widetilde{B}_\infty  \left(\theta \hh^{\frac \beta 2}  \underline{C}_{M,\beta}  \underline{\mu}^*  \right)^{-(p-1)}$.
  Owing to Lemma \ref{lemma::N_convergence},  up to reducing $\bar \varepsilon$, we have
  
  \begin{equation}
	  \label{eq::1_on_N}
	  \forall \varepsilon \in (0, \bar \varepsilon], \quad
	  \frac{1}{N(\varepsilon)} \leq \frac{2}{C_\beta} \varepsilon^2
	  \begin{cases}
		  1                                    & \mbox{if } \beta > 1, \\
		  R(\varepsilon)^{-1}                  & \mbox{if } \beta = 1, \\
		  M^{-\frac{1-\beta}{2}R(\varepsilon)} & \mbox{if } \beta < 1. 
	  \end{cases}
  \end{equation}
  Moreover,
  
  \begin{equation*}
	  \sum_{j=2}^{R(\varepsilon)} M^{\frac{1-\beta}{2}(j-1)} \leq 
	  \begin{cases}
		  \frac{1}{1 - M^{\frac{1-\beta}{2}}}                                  & \mbox{if } \beta > 1, \\
		  R(\varepsilon)                                                       & \mbox{if } \beta = 1, \\
		  \frac{M^{\frac{1-\beta}{2}R(\varepsilon)} }{M^{\frac{1-\beta}{2}}-1} & \mbox{if } \beta < 1. 
	  \end{cases}
  \end{equation*}
  Then
  $\displaystyle \left( \frac{1}{N(\varepsilon)}  \sum_{j=2}^{R(\varepsilon)} M^{\frac{1-\beta}{2}(j-1)} \right)^{\frac p2} \leq K_1 \varepsilon^p $
  with
  \begin{equation*}
	  K_1 = K_1(M, \beta, p) = \left(\frac{2}{C_\beta}\right)^{\frac p2}
	  \begin{cases}
		  \left(1-M^{\frac{1-\beta}{2}} \right)^{-\frac p2} & \mbox{if } \beta > 1, \\
		  1                                                 & \mbox{if } \beta = 1, \\
		  \left(M^{\frac{1-\beta}{2}}-1 \right)^{-\frac p2} & \mbox{if } \beta < 1. 
	  \end{cases}
  \end{equation*}
  We establish now  that the same inequality holds for $\displaystyle \frac{1}{N(\varepsilon)^{p-1}}  \sum_{j=2}^{R(\varepsilon)}  M^{\frac{p - (\beta+1)}{2}(j-1)}$.
  
  We take $\varepsilon \in (0, \bar{\varepsilon}\wedge 1]$.
  Since $R(\varepsilon) = O\left(\sqrt{\log(1/\varepsilon)}\right) = o\Big( \log(1/\varepsilon)\Big)$ as $\varepsilon\to0$ owing to \eqref{eq::R_ML2R}, then
  
  \begin{equation}
	  \label{eq::o_piccolo}
	  \forall \delta>0, \;\exists \,\varepsilon_\delta :\quad \forall \varepsilon\in(0,\min(\bar{\varepsilon}, \varepsilon_\delta,1)] \quad R(\varepsilon)\leq \delta \log\left( \frac1\varepsilon\right).
  \end{equation}
  
  \begin{itemize}
	  \item[$\bullet$] For  $p<\beta+1$ (so that $\beta >1$): \\
		Since $\sum_{j=2}^{R(\varepsilon)}  M^{\frac{p - (\beta+1)}{2}(j-1)} \leq \sum_{j\geq0}  M^{\frac{p - (\beta+1)}{2}j}  < +\infty$ and, owing to \eqref{eq::1_on_N}, $1/N(\varepsilon) \leq (2/C_\beta)\varepsilon^2$, it follows directly that, since $p\geq 2$,
		
		$$\frac{1}{N(\varepsilon)^{p-1}} \sum_{j=2}^{R(\varepsilon)}  M^{\frac{p - (\beta+1)}{2}(j-1)} \leq K_2 \varepsilon^{2(p-1)} \leq K_2 \varepsilon^p$$
		with $K_2 = K_2(M,\beta,p) = \left( 1-M^{\frac{p-(\beta+1)}{2}}\right)^{-1}  (2/C_\beta)^{p-1}$.
		  
	  \item[$\bullet$] For $p=\beta+1$ (so that $\beta \geq1$): 
		  
		$$\frac{1}{N(\varepsilon)^{p-1}} \sum_{j=2}^{R(\varepsilon)}  M^{\frac{p - (\beta+1)}{2}(j-1)} \leq \frac{R(\varepsilon)}{N(\varepsilon)^{p-1}} .$$
		  
		The case $\beta=1$ leads to $p=2$, for which the SLLN follows directly from \eqref{eq::L2_convergence}.
		If $\beta >1$, owing to the expression of $N$ given in \eqref{eq::1_on_N} and setting $\delta=1$ in inequality \eqref{eq::o_piccolo}, we get
		  
		$$\frac{R(\varepsilon)}{N(\varepsilon)^{p-1}} \leq  \left( \frac 2 {C_\beta}\right)^{p-1} \varepsilon^{2(p-1)} \log\left( \frac 1 \varepsilon\right).$$
		Since $p>2$, we have $\displaystyle 0 < \varepsilon^{2(p-1)} \log\left( \frac 1\varepsilon\right) \leq \varepsilon^p \sup_{\varepsilon\in(0,\min(\bar{\varepsilon}, \varepsilon_\delta,1)]} \left(\varepsilon^{p-2} \log\left(\frac 1 \varepsilon\right)\right)$. 
		Hence
		$$\frac{R(\varepsilon)}{N(\varepsilon)^{p-1}} \leq K_2 \varepsilon^p$$
		with  $ \displaystyle K_2 = K_2(M,\beta,p) = \left( \frac 2 {C_\beta}\right)^{p-1} \sup_{\varepsilon\in(0,\min(\bar{\varepsilon}, \varepsilon_\delta,1)]} \left(\varepsilon^{p-2} \log\left(\frac 1 \varepsilon\right)\right)$.
		  
	  \item[$\bullet$] For $p > \beta + 1$: As $p-(\beta+1) >0$, one has
		  
		$$\frac{1}{N(\varepsilon)^{p-1}} \sum_{j=2}^{R(\varepsilon)}  M^{\frac{p - (\beta+1)}{2}(j-1)} \leq \frac{1}{M^{\frac{p-(\beta +1)}{2}}-1} \frac{M^{\frac{p - (\beta+1)}{2}R(\varepsilon)}}{N(\varepsilon)^{p-1}}. $$
		Owing to \eqref{eq::1_on_N}, for all $\beta>0$,
		$$\frac{1}{N(\varepsilon)^{p-1}} \leq \left(\frac{2}{C_\beta} \right)^{p-1} \varepsilon^{2(p-1)}.$$
		We set $\delta = (p-2)/\left(\log(M)\frac{p-(\beta+1)}{2}\right)$ in \eqref{eq::o_piccolo}. 
		The case $p=2$ follows from \eqref{eq::L2_convergence}, therefore we can assume $p>2$, which guarantees $\delta>0$.
		Finally one has
		  
		$$\frac{M^{\frac{p - (\beta+1)}{2}R(\varepsilon)}}{N(\varepsilon)^{p-1}} 
		\leq \left( \frac{2}{C_\beta}\right)^{p-1} \varepsilon^{2(p-1)} \varepsilon^{-(p-2)} = \left( \frac{2}{C_\beta}\right)^{p-1}  \varepsilon^p
		$$
		which yields
		  
		$$\frac{1}{N(\varepsilon)^{p-1}} \sum_{j=2}^{R(\varepsilon)}  M^{\frac{p - (\beta+1)}{2}(j-1)} \leq K_2 \varepsilon^p$$
		with $K_2 = K_2(M,\beta,p) = \frac{1}{M^{\frac{p-(\beta +1)}{2}}-1} \left( \frac{2}{C_\beta}\right)^{p-1}$.
		
  \end{itemize}
  
  Then \eqref{eq::L_p_convergence} holds with $K(M,\beta,p) = C_p \Big(A_1(M,\beta,p) K_1(M,\beta,p)+ A_2(M,\beta,p) K_2(M,\beta,p)\Big)$. 
  
  %  Thanks to Markov's inequality and Borel-Cantelli's Lemma, \eqref{eq::slln_2} holds, which completes the proof.
  
  \
  
  $\rhd$ MLMC: The proof for the MLMC estimator follows the same steps, while replacing
  $\W_j^{R(\varepsilon)} = 1$, for $j=1,\ldots, R(\varepsilon)$, and $a_\infty\widetilde B_\infty=1$. 
  The only significantly different computations are the ones to get 
  \begin{equation}
   \label{eq::mlmc_ssln_inequality}
   \frac{1}{N(\varepsilon)^{p-1}} \leq \left(\frac{2}{C_\beta} \right)^{p-1} \varepsilon^{2(p-1)}.
  \end{equation}
  We give the detail of these computations in Appendix \ref{appendix::slln_mlmc_computations}.
\end{proof}

The Strong Law of Large Numbers follows as a consequence of Proposition \ref{prop::L_p_convergence}.

\begin{proof}[Proof of Theorem \ref{th::slln}]
	Owing to the decomposition \eqref{eq::I_1_I_2}, \eqref{eq::slln} amounts to proving
	
	\begin{equation}
		\label{eq::slln_2}
		\widetilde{I}_{\varepsilon_k}^1 \xrightarrow{a.s.} 0 \quad \mbox{ as } \quad k\to+\infty
		\quad \mbox{ and } \quad
		\widetilde{I}_{\varepsilon_k}^2 \xrightarrow{a.s.} 0 \quad \mbox{ as } \quad k\to+\infty.
	\end{equation}
	
	As $\lim_{\varepsilon\to 0} N_1(\varepsilon) = +\infty$ and the $(Y_{\hh}^{(1),k})_{k\geq 1}$ are \iid 
	and do not depend on $\varepsilon$, the convergence of $\widetilde{I}_{\varepsilon_k}^1$ is a direct application of the classical Strong Law of Large Numbers,
	for both ML2R and MLMC estimators.
	\
	
	To establish the $a.s.$ convergence of $\widetilde{I}_{\varepsilon_k}^2$, owing to Lemma \ref{prop::L_p_convergence}
	it is straightforward that for all sequence of positive values $(\varepsilon_k)_{k\geq1}$ such that $\varepsilon_k \to0$ as $k\to+\infty$ and $\sum_{k\geq1} \varepsilon_k^p <+\infty$
	 
	$$\sum_{k\geq1} \esp{\left| \widetilde{I}_{\varepsilon_k}^2 \right|^p}   < +\infty.$$
	Hence, by Beppo-Levi's Theorem, 
	$\sum_{k\geq1} \left| \widetilde{I}_{\varepsilon_k}^2 \right|^p   < +\infty \; a.s.$,
	which in turn implies 
	$\widetilde{I}_{\varepsilon_k}^2 \xrightarrow{a.s.} 0 \; \mbox{ as } k\to+\infty$.		
\end{proof}

\subsection{Proof of Central Limit Theorem}
\label{section::tcl}

This section is devoted to the proof of Theorems \ref{th::clt_g1} and \ref{th::clt_le1}. 
In order to satisfy a Lindeberg condition, we will need the assumption $\pa[1]{Z(h)}_{h \in \Hr}$ is $L^2$--uniformly integrable.
Owing to \eqref{lemaireweak_error}, $\bar R=1$, 
\begin{equation*}
    \abs{\esp{Z(h)}} = \abs{c_1 (1 - M^{\alpha})} h^{\alpha - \frac{\beta}{2}} + o(h^{\alpha - \frac{\beta}{2}}).
\end{equation*}
%$$M^{\frac \beta 2(j-1)} \left| \esp{Y_{\frac{\hh}{n_j}} - Y_{\frac{\hh}{n_{j-1}}}} \right| \leq M^{\left(\frac \beta 2-\alpha\right)(j-1)} c_1 \hh^{\alpha}\left( M^{\alpha } -1 \right) + o\left( M^{\left(\frac \beta2 -\alpha\right)(j-1)}\right).$$
Since $2\alpha\geq \beta$, this deterministic sequence $\pa[1]{\esp{Z(h)}}_{h \in \Hr}$ is bounded. Hence, the $L^2$--uniform integrability of $\pa{Z(h)}_{h \in \Hr}$ yields the $L^2$--uniform integrability of the centered sequence $\pa[1]{\widetilde Z(h)}_{h \in \Hr} = \pa{Z(h) - \esp{Z(h)}}_{h \in \Hr}$.

One criterion to verify the $L^2$--uniform integrability is the following.
\begin{Lemma}\mbox{}
	\label{prop::L_2_norm_unif_integrability}
	\begin{itemize}
		\item[$(a)$] If there exists a $p>2$ such that $\sup_{h \in \Hr} \normLp{p}{Z(h)} < +\infty$ the family $\pa[1]{Z(h)}_{h \in \Hr}$ is $L^2$--uniformly integrable.
		\item[$(b)$] If there exists a random variable $D^{(M)} \in L^2$ %$\widetilde Y_0^{(M)}\in\mathbf L^2(\PP)$ 
such that, as $h\to0$, 
\begin{equation*}
    Z(h) %= \pa{\frac{h}{M}}^{-\frac{\beta}{2}} \pa{Y_{\frac{h}{M}} - Y_h} 
\xrightarrow{\mathcal L} D^{(M)}
\end{equation*}
		      then the following conditions are equivalent (see~\cite{Bi99}, Theorem 3.6):
		      \begin{itemize}
		      \item[(i)] The family $\pa[1]{Z(h)}_{h \in \Hr}$ is $L^2$--uniformly integrable.
		      \item[(ii)] $\lim_{h \to 0} \normLp{2}{Z(h)} = \normLp{2}{D^{(M)}}$.
		      \end{itemize}
	\end{itemize}
\end{Lemma}

%\begin{proof}
	%$(a)$ We take $p'=p/2$ and notice that $p'>1$ and 
	%$$\sup_{h\in\mathcal H}\left\| \left(h^{-\frac \beta 2} \left( Y_h - Y_{\frac hM} \right) \right)^2 \right\|_{p'} 
	%= \sup_{h\in\mathcal H} \left(\left\| h^{-\frac \beta 2} \left( Y_h - Y_{\frac hM} \right)  \right\|_p\right)^2<+\infty.$$
	  
	%$(b)$ We set $X_n = h^{-\frac\beta2} \left( Y_h - Y_{\frac hM}\right)$ and we suppose that $X_n\xrightarrow{\mathcal L} X$ as $n\to+\infty$. 
	%Due to  Theorem 3.5 in~\cite{Bi99}, if $(X_n)_{n\geq1}$ is uniformly integrable,
	%then $X$ is integrable and $\esp{X_n} \to \esp{X}$ as $n\to+\infty$. 
	%The converse is true if $X_n$ and $X$ are non-negative
	%(see~\cite{Bi99}, Theorems 3.6).
	%We conclude noticing that the function $(\cdot)^2$ is continuous (hence the convergence in law is preserved) and non-negative.
%\end{proof}
Now we are in position to prove the Central Limit Theorem, in both cases $\beta>1$ and $\beta\in(0,1].$ 

\begin{proof}[Proof of Theorems \ref{th::clt_g1} and \ref{th::clt_le1}]
	
	Owing to the decomposition \eqref{eq::I_1_I_2}  (with $\W_j^{R(\varepsilon)} = 1$, $j=1,\ldots, R(\varepsilon)$ for MLMC estimator)  
	$$\frac{I_\pi^N(\varepsilon)-I_0}{\varepsilon} = \frac{\widetilde I_\varepsilon^1}{\varepsilon} + \frac{\widetilde I_\varepsilon^2}{\varepsilon} + \frac{\mu(\hh,R(\varepsilon),M)}{\varepsilon}$$
	where $\widetilde{I}_\varepsilon^1$ and $\widetilde{I}_\varepsilon^2$ are independent.
	The bias term has already been treated in \eqref{eq::bias_on_epsilon_lim}.\\
	
	\noindent $\rhd$ ML2R: \eqref{eq::clt_ml2r_g1} and \eqref{eq::clt_ml2r_le1} amounts to proving, as $\varepsilon\to0$, 
	
	\begin{minipage}{0.45\linewidth}
	\begin{equation}
		\label{eq::clt_1}
		\sqrt{N(\varepsilon)} \widetilde{I}_\varepsilon^1 \xrightarrow{\calL} \calN \pa{0, \frac{\Var(Y_\hh)}{\mathbf{q}_1}} 
	\end{equation}
	\end{minipage}
	\begin{minipage}{0.45\linewidth}
	\begin{equation}
		\label{eq::clt_2} \mbox{and} \qquad
		\frac{\widetilde{I}_\varepsilon^2}{\varepsilon} \xrightarrow{\calL} \calN \pa{0, \sigma_2^2} 
	\end{equation}
	\end{minipage}
	
	\noindent with $\sigma_2 = \sigma$ for $\beta\in(0,1]$. Indeed, for \eqref{eq::clt_1} let us write $\ds \frac{\widetilde{I}_\varepsilon^1}{\varepsilon} = \frac{1}{\varepsilon \sqrt{N(\varepsilon)}} \sqrt{N(\varepsilon)} \widetilde{I}_\varepsilon^1$. Using Lemma \ref{lemma::N_convergence}, $N(\varepsilon)$ reads
	$$
	N(\varepsilon) \stackrel{\varepsilon\to 0}{\sim} C_\beta \varepsilon^{-2}\left\{\begin{array}{ll} 1 & \mbox{ if }\beta>1,\\
	R(\varepsilon) & \mbox{ if } \beta =1,\\
	M^{\frac{1-\beta}{2} R(\varepsilon)} &\mbox{ if } \beta <1.
	\end{array} \right. \quad (ML2R)
	$$
	In particular, since $R(\varepsilon)\to+\infty$ as $\varepsilon\to0$, when $\beta\leq1$, $\displaystyle \frac{1}{\sqrt{N(\varepsilon)}} = o( \varepsilon)$ and the term $\displaystyle \frac{\widetilde{I}_\varepsilon^1}{\varepsilon}\to0$ in probability. Since $Y_{\hh}^{(1),k}$ does not depend on $\varepsilon$, $N_1(\varepsilon) \to +\infty$ and  
	$N_1(\varepsilon)/N(\varepsilon) \to \mathbf{q}_1$ as $\varepsilon\to0$, the asymptotic behaviour of the first term is driven by a regular Central Limit Theorem at rate $\sqrt{N(\varepsilon)}$, \emph{i.e.} 
	
	$$\sqrt{N(\varepsilon)} \widetilde{I}_\varepsilon^1 = \sqrt{N(\varepsilon)} \left[\frac{1}{N_1(\varepsilon)} \sum_{k=1}^{N_1(\varepsilon)} \pa{Y_{\hh}^{(1),k} -\esp{Y_{\hh}^{(1),k}}}\right] \xrightarrow{\varepsilon \to 0} \calN \pa{0, \frac{ \Var\pa{Y_{\hh}}}{\mathbf{q}_1}}$$
	which proves \eqref{eq::clt_1}.
	
	We will use Lindeberg's Theorem for triangular arrays of martingale increments (see Corollary 3.1 p.58 in~\cite{HaHe80}) to establish \eqref{eq::clt_2}. The random variables $\widetilde{Y}_j^k$ being centered and independent, the variance reads
	
	\begin{equation*}
		\Var\left( \sum_{j=2}^{R(\varepsilon)} \sum_{k=1}^{N_j(\varepsilon)} \frac 1 \varepsilon \frac{\W_j^{R(\varepsilon)}}{N_j(\varepsilon)} \widetilde{Y}_j^k \right) = \frac 1 {\varepsilon^2} \sum_{j=2}^{R(\varepsilon)} \left( \frac{\W_j^{R(\varepsilon)}}{N_j(\varepsilon)}\right)^2 N_j(\varepsilon) \Var\left(\widetilde{Y}_j\right)
		=\frac{1}{\varepsilon^2}  \sum_{j=2}^{R(\varepsilon)} \frac{\left( \W_j^{R(\varepsilon)}\right)^2}{N_j(\varepsilon)} \Var\left( \widetilde{Y}_j \right).                                                                                                                                                                                           
	\end{equation*}
	Noticing that $0\leq \frac1x-\frac{1}{\lceil x\rceil} \leq \frac{1}{x^2}$, $x>0$, 
	and that $N_j(\varepsilon) = \lceil q_j(\varepsilon) N(\varepsilon)\rceil$, we derive 
	
	\begin{equation*}
		\left|  \frac{1}{\varepsilon^2}\sum_{j=2}^{R(\varepsilon)} \frac{\left( \W_j^{R(\varepsilon)}\right)^2}{N_j(\varepsilon)} \Var\left( \widetilde{Y}_j \right) -  \frac{1}{\varepsilon^2} \sum_{j=2}^{R(\varepsilon)} \frac{\left( \W_j^{R(\varepsilon)}\right)^2}{q_j(\varepsilon) N(\varepsilon)} \Var\left( \widetilde{Y}_j \right)\right| 
		\leq \frac{1}{\varepsilon^2} \sum_{j=2}^{R(\varepsilon)} \frac{\left( \W_j^{R(\varepsilon)}\right)^2}{\left(q_j(\varepsilon) N(\varepsilon) \right)^2} \Var\left( \widetilde{Y}_j \right) . 
	\end{equation*}
	The conclusion will follow from 

	\begin{equation}
		\label{eq::variance}
		\lim_{\varepsilon\to0} \frac{1}{\varepsilon^2} \sum_{j=2}^{R(\varepsilon)} \frac{\left( \W_j^{R(\varepsilon)}\right)^2}{N(\varepsilon)q_j(\varepsilon)} \Var\left( \widetilde{Y}_j \right) = \sigma_2^2
	\end{equation}
	and
	\begin{equation}
		\label{eq::ceil}
		\lim_{\varepsilon\to0} \frac{1}{\varepsilon^2} \sum_{j=2}^{R(\varepsilon)} \frac{\left( \W_j^{R(\varepsilon)}\right)^2}{\left(q_j(\varepsilon) N(\varepsilon)\right)^2} \Var\left( \widetilde{Y}_j \right) = 0.
	\end{equation}
	
	Owing to the definition of $Z_j$ given in \eqref{def::tilde_Z_j}, we get $\Var(\widetilde{Y}_j) = \left( \frac{\hh}{n_j} \right)^\beta \Var(Z_j)$ and, using the expression of $q_j(\varepsilon)$ given in \eqref{eq::q_j}, we obtain
	
	\begin{equation*}
		\frac{1}{\varepsilon^2} \sum_{j=2}^{R(\varepsilon)} \frac{\left( \W_j^{R(\varepsilon)}\right)^2}{N(\varepsilon)q_j(\varepsilon)} \Var\left( \widetilde{Y}_j \right) = \frac{1}{\varepsilon^2 N(\varepsilon)}\frac{\mathbf{h}^{\frac \beta 2} }{\theta \underline{C}_{M,\beta}\mu^*(\varepsilon)} \sum_{j=2}^{R(\varepsilon)} \left| \W_j^{R(\varepsilon)} \right| M^{\frac{1-\beta}{2}(j-1) } \Var(Z_j).
	\end{equation*}
	$\bullet \; Case \; \beta>1$:
	Owing to the expression of $N(\varepsilon)$ given in Lemma \ref{lemma::N_convergence} when $\beta>1$,
	$$\lim_{\varepsilon\to0} \frac{1}{\varepsilon^2 N(\varepsilon)}\frac{\mathbf{h}^{\frac \beta 2} }{\theta \underline{C}_{M,\beta}\mu^*(\varepsilon)} = \frac1{\Sigma} \frac{\mathbf{h}^{\frac \beta 2}}{\sqrt{\Var (Y_0) V_1} \underline{C}_{M,\beta} }$$
	and owing to the limit in Lemma \ref{lemma::W_j_R_limits} $(d)$ with $\gamma = \frac{1-\beta}{2}<0$,
	$$\sum_{j=2}^{R(\varepsilon)} \left| \W_j^{R(\varepsilon)} \right| M^{\frac{1-\beta}{2}(j-1) } \Var(Z_j) = \sum_{j=2}^{+\infty}  M^{\frac{1-\beta}{2}(j-1) } \Var(Z_j)<+\infty.$$
	Hence the convergence of the variance \eqref{eq::variance} holds for Theorem \ref{th::clt_g1}.\\
	$\bullet \; Case \; \beta\leq1$: 
	Owing to the expression of $N(\varepsilon)$ given in Lemma \ref{lemma::N_convergence} when $\beta\leq1$, we get, as $\varepsilon\to0$,
	$$\frac{1}{\varepsilon^2 N(\varepsilon)}\frac{\mathbf{h}^{\frac \beta 2} }{\theta \underline{C}_{M,\beta}\mu^*(\varepsilon)} \sim \frac{1}{V_1 \pa{1+M^{\frac{\beta}{2}}}^2} \begin{cases} (R(\varepsilon))^{-1} & \mbox{if } \beta=1, \\ \pa{M^{\frac{1-\beta}{2}R(\varepsilon)} a_\infty \sum_{j\geq1} \left| \sum_{\ell = 0}^{j-1} b_\ell \right| M^{\frac{\beta-1}{2}j}}^{-1} & \mbox{if } \beta<1.\end{cases}$$
	We notice that $\ds \lim_{j\to+\infty} \Var(Z_j) = \tilde v_\infty (M,\beta)$ if $2\alpha>\beta$ and $\ds \lim_{j\to+\infty} \Var(Z_j) = \tilde v_\infty (M,\beta)- c_1^2 \left(1 - M^{\frac \beta2} \right)^2$ if $2\alpha=\beta$.
	Hence, owing to the limit in Lemma \ref{lemma::W_j_R_limits} $(d)$ with $\gamma = \frac{1-\beta}{2}\geq0$, we obtain \eqref{eq::variance} with $\sigma_2 = \sigma$ given in \eqref{eq::sigma_clt_le1} in Theorem \ref{th::clt_le1}. 
	
	For \eqref{eq::ceil}, it follows from the expression of $q_j(\varepsilon)$ in \eqref{eq::q_j} that
	$\displaystyle\frac{|\W_j^{R(\varepsilon)}|}{q_j(\varepsilon)} = \frac{M^{\frac{\beta+1}{2}(j-1)}}{\theta \hh^{\frac \beta 2}  \underline{C}_{M,\beta}  \mu^*(\varepsilon)}$. Owing to the definition of $Z_j$ in \eqref{def::tilde_Z_j} and to inequality \eqref{rmk::mu_bounds}, we get
	\begin{align*}
		\frac{1}{\varepsilon^2} \sum_{j=2}^{R(\varepsilon)} \frac{\left( \W_j^{R(\varepsilon)}\right)^2}{\left(q_j(\varepsilon) N(\varepsilon)\right)^2} \Var\left( \widetilde{Y}_j \right) 
		  & \leq \frac{\hh^\beta}{\left(\theta \hh^{\frac \beta 2}  \underline{C}_{M,\beta}  \underline{\mu}^*\right)^2} \frac{1}{\left( \varepsilon N(\varepsilon)\right)^2} \sum_{j=2}^{R(\varepsilon)} M^{(j-1)}\Var(Z_j)                        \\
		  & \leq \frac{\hh^\beta}{\left(\theta \hh^{\frac \beta 2}  \underline{C}_{M,\beta}  \underline{\mu}^*\right)^2} \frac{1}{\left( \varepsilon N(\varepsilon)\right)^2} \left(\sup_{j\geq1}\Var(Z_j) \right)\frac{M^{R(\varepsilon)-1}}{M-1}. 
	\end{align*}
	We conclude by showing that 
	\begin{equation}
	\label{eq::zero_limit_in_tcl}
	 \frac{M^{R(\varepsilon)}}{(\varepsilon N(\varepsilon))^2} \to 0, \quad \mbox{as} \quad \varepsilon\to0.
	\end{equation}
	Owing to the expression of $R(\varepsilon)$ given in \eqref{eq::R_ML2R}, we notice that 
	$R(\varepsilon) = O\left(\sqrt{\log(1/\varepsilon)}\right) = o\Big(\log(1/\varepsilon)\Big)$ as $\varepsilon\to0$. 
	Moreover, using Lemma \ref{lemma::N_convergence}, up to another reduction of $\bar \varepsilon$, we have 
	$\frac{1}{(\varepsilon N(\varepsilon))^2} \leq \left(\frac{2}{C_\beta}\right)^2 \varepsilon^2$ for all $\beta>0$. This in turn yields
	$$\frac{M^{R(\varepsilon)}}{(\varepsilon N(\varepsilon))^2} \leq \left(\frac{2}{C_\beta}\right)^2 \varepsilon^2M^{R(\varepsilon)} = C_\beta^{-2} e^{\log(M)R(\varepsilon) - 2\log(1/\varepsilon)}\to0, \quad \mbox{as} \quad\varepsilon\to0.$$
	Then \eqref{eq::ceil} is proved and so is the first condition of Lindeberg's Theorem.
	
	\
	%================================================================================================================
	% second condition
	%================================================================================================================
	
	For the second condition of Lindeberg's Theorem we need to prove that, for every $\eta >0$, 
	
	\begin{equation}
		\label{eq::second_lindeberg_condition}
		\sum_{j=2}^{R(\varepsilon)} \sum_{k=1}^{N_j(\varepsilon)} \esp{\left( \frac 1 \varepsilon \frac{\W_j^{R(\varepsilon)}}{N_j(\varepsilon)}  \widetilde{Y}_j^k \right)^2 \boldsymbol{1}_{\left\{  \left| \frac 1 \varepsilon \frac{\W_j^{R(\varepsilon)}}{N_j(\varepsilon)}   \widetilde{Y}_j^k \right|>\eta  \right\} } } \xrightarrow{\varepsilon \to 0} 0.
	\end{equation} 
	Since the $\left(\widetilde{Y}_j^k\right)_{k=1,\ldots, N_j(\varepsilon)}$ are identically distributed, we can write
	
	\begin{equation*}
		\sum_{j=2}^{R(\varepsilon)} \sum_{k=1}^{N_j(\varepsilon)}  \esp{\left( \frac 1 \varepsilon \frac{\W_j^{R(\varepsilon)}}{N_j(\varepsilon)}  \widetilde{Y}_j^k \right)^2 \boldsymbol{1}_{\left\{  \left| \frac 1 \varepsilon \frac{\W_j^{R(\varepsilon)}}{N_j(\varepsilon)}   \widetilde{Y}_j^k \right|>\eta  \right\} } }                                                                       \leq \sum_{j=2}^{R(\varepsilon)} \frac{1}{\varepsilon^{2}} \frac{\left|\W_j^{R(\varepsilon)}\right|^{2}}{N_j(\varepsilon)}   \esp{ \left(   \widetilde{Y}_j \right)^{2} \boldsymbol{1}_{\left\{  \left| \widetilde{Y}_j \right|> \eta \, \varepsilon  \frac{N_j(\varepsilon)}{\left| \W_j^{R(\varepsilon)} \right|}  \right\}  }}. 
	\end{equation*}
	We set $\widetilde Z_j = Z_j - \esp{Z_j}$. Replacing $q_j$ by its values given in \eqref{eq::q_j}, using Inequality  \eqref{eq::sup_W} from Lemma \ref{lemma::W_j_R_limits} $(b)$ and 
	the elementary  inequality $\frac{1}{N_j(\varepsilon)}\leq \frac{1}{q_j(\varepsilon)N(\varepsilon)}$, yields
	
	\begin{align*}
		  & \sum_{j=2}^{R(\varepsilon)} \sum_{k=1}^{N_j(\varepsilon)} \esp{\left( \frac 1 \varepsilon \frac{\W_j^{R(\varepsilon)}}{N_j(\varepsilon)}  \widetilde{Y}_j^k \right)^2 \boldsymbol{1}_{\left\{  \left| \frac 1 \varepsilon \frac{\W_j^{R(\varepsilon)}}{N_j(\varepsilon)}   \widetilde{Y}_j^k \right|>\eta  \right\} } }                                                                                                                                                                      \\
		  & \leq \frac{1}{\theta \hh^{\frac \beta 2} \underline{C}_{M,\beta} \mu^*(\varepsilon)}  \frac{1}{\varepsilon^2 N(\varepsilon)} \sum_{j=2}^{R(\varepsilon)}  \left|\W_j^{R(\varepsilon)}\right|  M^{\frac{\beta+1}{2}(j-1)}  \esp{\left( \widetilde{Y}_j \right)^{2} \boldsymbol{1}_{\left\{  \left| \widetilde{Y}_j \right|> \eta \frac{\theta \mathbf{h}^{\frac \beta 2} \underline{C}_{M,\beta} \mu^*(\varepsilon)}{M^{\frac{\beta+1}{2}(j-1)}} \varepsilon N(\varepsilon) \right\}}} \\
		  & \leq \frac{\hh^{\frac \beta 2} }{\theta \underline{C}_{M,\beta} \underline{\mu}^*} a_\infty \widetilde{B}_\infty  \frac{1}{\varepsilon^2 N(\varepsilon)}  \sum_{j=2}^{R(\varepsilon)} M^{\frac{1-\beta}{2}(j-1)} \esp{\left( \widetilde Z_j \right)^{2} \boldsymbol{1}_{\left\{  \left| \widetilde Z_j \right|> \eta \theta \underline{C}_{M,\beta} \underline{\mu}^* \varepsilon N(\varepsilon) M^{-\frac{j-1}{2}}\right\}}}                                                   \\
		  & \leq \frac{\hh^{\frac \beta 2} }{\theta \underline{C}_{M,\beta} \underline{\mu}^*} a_\infty \widetilde{B}_\infty  \sup_{2\leq j\leq R(\varepsilon)} \esp{\left( Z_j \right)^{2} \boldsymbol{1}_{\left\{  \left| \widetilde Z_j \right|> \Theta \varepsilon N(\varepsilon) M^{-\frac{R(\varepsilon)}{2}}\right\}}} \frac{1}{\varepsilon^2 N(\varepsilon)} \sum_{j=2}^{R(\varepsilon)} M^{\frac{1-\beta}{2}(j-1)}                                                                                            
	\end{align*}
	where we set $\Theta = \eta \theta \underline{C}_{M,\beta} \underline{\mu}^* \sqrt{M}$.
	Now, it follows from Lemma \ref{lemma::N_convergence}  that 
	
	$$ \frac{1}{\varepsilon^2 N(\varepsilon)} \sum_{j=2}^{R(\varepsilon)} M^{\frac{1-\beta}{2}(j-1)} 
	= \frac{1}{\varepsilon^2 N(\varepsilon)} \left( \frac{M^{\frac{1-\beta}{2}R(\varepsilon)}-M^{\frac{1-\beta}{2}}}{M^{\frac{1-\beta}{2}}-1} \ind{\beta\neq1} + (R(\varepsilon)+1)\ind{\beta=1} \right) \to K,$$
	as $\varepsilon\to0$, where $K$ is a real positive constant. 
	Owing to \eqref{eq::zero_limit_in_tcl} $\displaystyle\lim_{\varepsilon\to0} \varepsilon N(\varepsilon) M^{-\frac{R(\varepsilon)}{2}} = +\infty$.
	Hence, since we assumed that the family $(Z_j)_{j\geq 1}$ is $L^2$--uniformly integrable, we obtain that
	
	\begin{equation}
		\label{eq::sup_lim_uniform_integrability}
		\lim_{\varepsilon\to0} \sup_{2\leq j\leq R(\varepsilon)} \esp{\left( Z_j \right)^{2} \boldsymbol{1}_{\left\{  \left| Z_j \right|> \Theta \varepsilon N(\varepsilon) M^{-\frac{R(\varepsilon)-1}{2}}\right\}}} = 0
	\end{equation}
	and the second condition of Lindeberg's Theorem is proved.
	
	\
	
	\noindent $\rhd$ MLMC: 
	The proofs are quite the same as for ML2R, up to the constant $1+\frac1{2\alpha}$, coming from the constant $C_\beta$ in the asymptotic of $N(\varepsilon)$. Using Lemma \ref{lemma::N_convergence} and the expression of $R(\varepsilon)$ given in \eqref{eq::R_MLMC}, we obtain
	$$N(\varepsilon) \stackrel{\varepsilon\to 0}{\sim}  C_\beta \varepsilon^{-2} \left\{\begin{array}{ll} 1 & \mbox{ if }\beta >1,\\
	\frac{1}{\alpha\log(M)}\log\left(\frac 1 \varepsilon\right) &\mbox{ if } \beta =1,\\
	\varepsilon^{-\frac{1-\beta}{2\alpha}} &\mbox{ if } \beta <1.
	\end{array} \right. \quad (MLMC)
	$$
	We replace $\W_j^R = 1$, $j=1, \ldots, R$ and $a_\infty \widetilde B_\infty = 1$. The only significant difference comes when $\beta<1$, while proving \eqref{eq::zero_limit_in_tcl}. In this case, owing to Lemma \ref{lemma::N_convergence} as we did in \eqref{eq::1_on_N} and using the expression of $R(\varepsilon)$ given in \eqref{eq::R_MLMC}, up to reducing $\bar \varepsilon$, we can write
	$$
	\frac{M^{R(\varepsilon)}}{(\varepsilon N(\varepsilon))^2} 
	\leq \left(\frac{2}{C_\beta}\right)^2 \varepsilon^2 M^{-(1-\beta)R(\varepsilon)}M^{R(\varepsilon)}  
	\leq \left(\frac{2}{C_\beta}\right)^2 M^{\beta(C_R^{(1)} +1)} \varepsilon^{2-\frac\beta\alpha} 
	$$
	which goes to $0$, owing to the strict inequality assumption $2\alpha>\beta$.	
\end{proof}

\section{Applications}
\label{section::applications}
\subsection{Diffusions}

In this section we retrieve a recent result by Kebaier and Ben Alaya (see~\cite{BeKe15}) obtained for MLMC estimators 
and we extend it to the ML2R estimators and to the use of path-dependent functionals. Let $(X_t)_{t\in[0,T]}$ a Brownian diffusion process solution to the stochastic differential equation 
$$X_t = X_0 + \int_0^t b(s, X_s)ds + \int_0^t \sigma(s,X_s) \d W_s, \quad t\in[0,T]$$
where $b:[0,T]\times\RR^d \to \RR^d$, $\sigma:[0,T]\times\RR^d\to M(d,q,\RR)$ are continuous functions, 
Lipschitz continuous in $x$, uniformly in $t\in[0,T]$, $(W_t)_{t\in[0,T]}$ is a $q$-dimensional Brownian motion 
independent of $X_0$, both defined on a probability space $(\Omega, \mathcal A, \PP)$.

We know that $X=(X_t)_{t\in[0,T]}$ is the unique $(\mathcal F_t^W)_{t\in[0,T]}$-adapted solution to this equation, 
where $\mathcal F^W$ is the augmented filtration of $W$.
The process $(X_t)_{t\in[0,T]}$ cannot be simulated at a reasonable computational cost (at least in full generality), which leads to introduce some simulatable time discretization schemes, the simplest being undoubtedly the Euler scheme with step $h=\frac Tn$, $n\geq1$, defined by 
\begin{equation} \label{eq::diffusion}
	\bar X_t^n = X_0 + \int_0^tb(\underline s, \bar X_{\underline s}^n)ds + \int_0^t \sigma(\underline s, \bar X_{\underline s}^n) \d W_s
\end{equation}
with $\underline s = \frac{\lfloor ns\rfloor}{n}$, $s\in[0,T]$. In particular, if we set $t_k^n = k\frac Tn$,
$$\bar X_{t_{k+1}^n}^n = \bar X_{t_k^n}^n + b(t_k^n, \bar X_{t_k^n}^n) h + \sigma(t_k^n, \bar X_{t_k^n}^n)\sqrt h U_{k+1}^n,  \quad k\in \left\{ 0, \ldots,n-1\right\}$$
where $U_{k+1}^n = \frac{W_{t_{k+1}^n} - W_{t_k^n}}{\sqrt h}$ is \iid with distribution $\mathcal N(0,I_q)$. Furthermore, we also derive from \eqref{eq::diffusion} that
$$\bar X_t^n = \bar X_{\underline t}^n + b(\underline t, \bar X_{\underline t}^n) (t - \underline t) + \sigma(\underline t, \bar X_{\underline t}^n)(W_t-W_{\underline t}),  \quad t\in[0,T]. $$
It is classical background that, under the above assumptions on $b, \sigma, X_0$ and $W$, the Euler scheme satisfies 
the following a priori $L^p$-error bounds:
\begin{equation} \label{eq::diffusion_lp}
	\forall p \ge 2, \; \exists c_{b,\sigma,p,T}>0, \quad \left\|\sup_{t\in[0,T]} |X_t - \bar X_t^n|\right\|_p \leq c_{b,\sigma,p,T} \sqrt{\frac Tn}\left(1+\normLp{p}{X_0}\right).
\end{equation}

For the weak error expansion the existing results are less general. Let us recall as an illustration the celebrated Talay-Tubaro's and Bally-Talay's weak error expansions for marginal functionals of Brownian diffusions, i.e. functionals of the form $F(X) = f(X_T)$.

\begin{Theorem} 
	(a) Regular setting (Talay-Tubaro~\cite{TaTu90}): If $b$ and $\sigma$ are infinitely differentiable with bounded partial derivatives and if $f :  \RR^d \to \RR$ is an infinitely differentiable function, with all its partial derivatives having a polynomial growth, then for a fixed maturity $T > 0$ and for every integer $R \in N^*$
	\begin{equation} \label{eq::diffusion_weak_error}
		\esp{f(\bar X_T^n)} - \esp{f(X_T)} = \sum_{k=1}^R c_k \pa{\frac{1}{n}}^k + O\left( \pa{\frac{1}{n}}^{R+1} \right)
	\end{equation}
	where the coefficients $c_k$ depend on $b, \sigma, f , T$ but not on $n$.
	
	\noindent (b) (Hypo-)Elliptic setting (Bally-Talay~\cite{BaTa96}): If $b$ and $\sigma$ are infinitely differentiable with bounded partial derivatives and if $\sigma$ is uniformly elliptic in the sense that
	$$\forall x\in\RR^d, \; \forall t\in [0,T], \quad \sigma\sigma^*(x) \geq \varepsilon_0 I_q, \; \varepsilon_0>0$$
	or more generally if $(b, \sigma)$ satisfies the strong H\"ormander hypo-ellipticity assumption, then~\eqref{eq::diffusion_weak_error} holds true for every bounded Borel function $f : \RR^d \to \RR$.
\end{Theorem}

For more general path-dependent functionals, no such result exists in general. For various classes of specified functionals depending on the running maximum or mean, some exit stopping time, first order weak expansions in $h^\alpha, \alpha \in (0,1]$, have sometimes been established (see~\cite{LePa14} for a brief review in connection with multilevel methods). However, as emphasized by the numerical experiments carried out in~\cite{LePa14}, such weak error expansion can be highly suspected to hold at any order under reasonable smoothness assumptions.

In this section we consider $F:\mathcal{C}_b([0,T], \R^d) \to \R$ a Lipschitz continuous functional and we set 
\begin{equation*}
    Y_0 = F(X) \quad \text{and} \quad Y_h = F(\bar X^n) \text{ with $h = \frac{T}{n}$ and $n \ge 1$ (\emph{i.e.} $\hh = T$).}
\end{equation*}
We assume the weak error expansion \eqref{lemaireweak_error}. We prove now that both estimators ML2R~\eqref{intro:MLRR-estimator} and MLMC~\eqref{intro:MLMC-estimator} satisfy a Strong Law of Large Numbers and a Central Limit Theorem when $\varepsilon$ tends to 0.
 
\begin{Theorem} Let $X_0 \in L^2$ and assume that $F:\mathcal{C}_b([0,T], \R^d) \to \R$ is a Lipschitz continuous functional. Then the assumption~\eqref{lemairestrong_error} is satisfied with $\beta = 1$.

If $X_0 \in L^p$ for $p \ge 2$, then the $L^p$--strong error assumption $\normLp{p}{Y_h - Y_0} \le V_1^{(p)} \sqrt{h}$ is satisfied so that both ML2R and MLMC estimators satisfy Theorem~\ref{th::slln}.

If $X_0 \in L^p$ for $p > 2$ and if $F$ is differentiable with $DF$ continuous, then the sequence $\pa[1]{Z(h)}_{h \in \Hr}$ is $L^2$--uniformly integrable and  
\begin{equation} \label{eq:cond_m2_diff}
    \exists v_\infty > 0, \quad  \lim_{h \to 0} \normLp{2}{Z(h)}^2 = (M-1) v_\infty.
\end{equation}
As a consequence, both ML2R and MLMC estimators satisfy Theorem~\ref{th::clt_le1} (case $\beta = 1$).
\end{Theorem}

\begin{proof}
First, note that if $F$ is a Lipschitz continuous functional, with Lipschitz coefficient $[F]_{\Lip}$, we have for all $p \ge 2$ 
\begin{equation*}
    \normLp{p}{Y_h - Y_0}^p
 \le [F]_{\Lip}^p \esp{\sup_{t \in [0, T]} \abs{X_t - \bar X^n_t}^p} \\
 \le [F]_{\Lip}^p c_{b, \sigma, p, T}^p (1 + \normLp{p}{X_0})^p  h^{\frac{p}{2}},
\end{equation*}
then $(Y_h)_{h \in \Hr}$ satisfies~\eqref{lemairestrong_error} with $\beta = 1$ and the $L^p$--strong error assumption as soon as $X_0 \in L^p$. 

Assume now that $X_0 \in L^p$ for $p > 2$. By a straightforward application of Minkowski's inequality we deduce from the $L^p$--strong error assumption that $\normLp{p}{Y_{\frac{h}{M}} - Y_h} \le C \sqrt{h}$ and then that $\sup_{h \in \Hr} \normLp{p}{Z(h)} < +\infty$. Applying the criterion $(a)$ of Lemma~\ref{prop::L_2_norm_unif_integrability} we prove that $\pa[1]{Z(h)}_{h \in \Hr}$ is $L^2$--uniformly integrable. 

At this stage it remains to prove~\eqref{eq:cond_m2_diff}. The key is Theorem~3 in~\cite{BeKe15}, where it is proved that 
\begin{equation*}
    \sqrt{n M} \left( \bar X^n - \bar X^{nM}\right) \xrightarrow{stably} U^{(M)}, \quad \text{as $n \to +\infty$},
\end{equation*}
where $U^{(M)}=(U^{(M)}_t)_{t\in[0,T]}$ is the $d-$dimensional process satisfying 
\begin{equation} \label{eq:defU_diff} 
    U^{(M)}_t = \sqrt{\frac{M-1}{2}} \sum_{i,j=1}^q V_t \int_0^t (V_s)^{-1} \nabla \varphi_{.j}(X_s) \varphi_{.i}(X_s) \d B_s^{i,j}, \quad t\in[0,T].
\end{equation}
We recall the notations of Jacod and Protter~\cite{JaPr98}
$$dX_t = \varphi(X_t) dW_t = \sum_{j=0}^q \varphi_{.j}(X_t) \d W_t^j$$
with $\varphi_{.j}$ representing the $j$th column of the matrix $\varphi = [\varphi_{ij}]_{\substack{i=1,\ldots,d, \\j=1,\ldots,q}}$, for $j=1,\ldots, q$, $\varphi_0=b$ and $W_t:=(t,W_t^1, \ldots,W_t^q)'$ (column vector), where $W_t^0=t$ and the $q$ remaining components make up a standard Brownian motion. Moreover, $\nabla \varphi_{.j}$ is a $d\times d$ matrix where $(\nabla \varphi_{.j})_{ik} = \partial_{x^k} \varphi_{ij}$  (partial derivative of $\varphi_{ij}$ with respect to the $k$th coordinate) and $(V_t)_{t\in[0,T]}$ is the $\RR^{d\times d}$ valued process solution of the linear equation 
$$V_t = I_d + \sum_{j=0}^q \int_0^t \nabla \varphi_{.j}(X_s) \d W_s^jV_s, \quad t\in[0,T].$$
Here $(B^{ij})_{1\leq i,j\leq q}$ is a standard $q^2$-dimensional Brownian motion independent of $W$. This process is defined on an extension $(\widetilde \Omega, \widetilde{\mathcal F}, (\widetilde{\mathcal F_t})_{t\geq0}, \widetilde \PP)$ of the original space $(\Omega, \mathcal F, (\mathcal F_t)_{t\geq 0}, \PP)$ on which lives~$W$.

We write, using that $h = \frac{T}{n}$, 
\begin{equation*}
    Z(h) = \sqrt{n M} \left( F(\bar X^{nM}) - F(\bar X^{n})\right) = - \int_0^1 DF \left( u\bar X^n + (1-u) \bar X^{nM} \right) \d u \cdot U^{(M)}_n
\end{equation*}
where $U^{(M)}_n := \sqrt{n M} \left( \bar X^n - \bar X^{nM} \right)$. The function $(x_1,x_2,x_3) \mapsto \int_0^1 DF(ux_1+ (1-u)x_2) \d u x_3$ is continuous, and it suffices to prove that $\left(\bar X^{n}, \bar X^{nM}, U^{(M)}_n \right) \xrightarrow{\mathcal L} (X,X,U^{(M)})$, as $n$ goes to infinity, to conclude that 
\begin{equation} \label{eq:ZconvLaw_diff}
    Z(h) \xrightarrow{\mathcal L} - DF(X)U^{(M)}, \quad \text{as $h \to 0$}.
\end{equation}
Let two bounded Lipschitz continuous functionals be $\phi:\mathcal C_b([0,T], \R^{2d})\to \RR$ and $\psi:\mathcal C_b([0,T], \R^{d})\to \RR$ and let denote $\widetilde X^n = (\bar X^n, \bar X^{nM})$ and $\widetilde X = (X,X)$. We write $\esp{\phi(\widetilde X^n) \psi (U^{(M)}_n) - \phi(\widetilde X)\psi(U^{(M)})} = 
\esp{(\phi(\widetilde X ^n)-\phi(\widetilde X))\psi (U^{(M)}_n) + \phi(\widetilde X)(\psi(U^{(M)}_n)-\psi(U^{(M)}))} $. Since $(U^{(M)}_n)_{n \ge 1}$ converges stably with limit $U^{(M)}$, we have that $\lim_{n \to +\infty} \esp{\phi(\widetilde X)(\psi(U^{(M)}_n)-\psi(U^{(M)}))} = 0$. On the other hand, owing to \eqref{eq::diffusion_lp}, we prove that $\lim_{n \to +\infty} \esp{(\phi(\widetilde X^n)-\phi(\widetilde X))\psi(U^{(M)}_n)} = 0$.

By \eqref{eq:ZconvLaw_diff} and Lemma~\ref{prop::L_2_norm_unif_integrability} $(b)$ we have that $\lim_{h \to 0} \normLp{2}{Z(h)}^2 = \normLp{2}{DF(X) U^{(M)}}^2 = (M-1) v_\infty$ with $v_\infty = \normLp{2}{DF(X) \frac{U^{(M)}}{M-1}}^2$ which does not depend on $M$ owing to the definition of $U^{M}$ given in~\eqref{eq:defU_diff}.
\end{proof}

\subsection{Nested Monte Carlo}

The aim of a \emph{nested} Monte Carlo method is to compute by Monte Carlo simulation 
$$\esp{f(\espc{X}{Y})}$$
where $(X,Y)$ is a couple  of $\RR\times\RR^{q_Y}$-valued random variables defined on a probability space $(\Omega, \mathcal{A}, \PP)$ with $X \in L^2(\PP)$ and $f:\RR\to\RR$ is a Lipschitz continuous function with Lipschitz coefficient $\left[f\right]_{\Lip}$. We assume that there exists a Borel function $F:\RR^{q_\xi}\times \RR^{q_Y}\to\RR$ and a random variable $\xi:(\Omega,\mathcal{A}) \to\RR^{q_\xi}$ independent of $Y$ such that
$$X = F(\xi,Y)$$ 
and we set $\mathbf{h} = \frac{1}{K_0}$ for some integer $K_0 \ge 1$, $h= 1/K$, $K\!\in K_0 \N^* = \ac{K_0, 2K_0, \dots}$ and 

\begin{equation}
 Y_0 := f(\espc{X}{Y}), \quad 
 Y_h = Y_{\frac{1}{K}}:= f \left( \frac{1}{K} \sum_{k=1}^{K} F(\xi_k, Y) \right)
\end{equation}
where $(\xi_k)_{k\geq 1}$ is a sequence of \iid variables, $\xi_k \sim \xi$, independent of $Y$.  
A \emph{nested} ML2R estimator then writes ($n_j=M^{j-1}$)
\begin{multline}
 \label{eq::nested_ml2r}
 I_{\pi}^N = \frac{1}{N_1} \sum_{i=1}^{N_1} f \left( \frac{1}{K} \sum_{k=1}^{K} F\left(\xi_k^{(1),i}, Y^{(1),i}\right) \right) \\ + \sum_{j=2}^R \frac{\W_j^R}{N_j} \sum_{i=1}^{N_j} \left(  f \left( \frac{1}{n_jK} \sum_{k=1}^{n_jK} F\left(\xi_k^{(j),i}, Y^{(j),i}\right) \right) - f \left( \frac{1}{n_{j-1}K} \sum_{k=1}^{n_{j-1}K} F\left(\xi_k^{(j),i}, Y^{(j),i}\right) \right) \right)
\end{multline}
where $\left(Y^{(j),i}\right)_{i\geq1}$ is a sequence of independent copies of $Y^{(j)} \sim Y$, $j=1, \ldots,R$, $Y^{(j)}$ independent of $Y^{(\ell)}$ for $j\neq\ell$, and $\left(\xi_k^{(j),i}\right)_{k,i\geq1, j=1,\ldots,R}$ is a sequence of \iid variables $\xi_k^{(j),i} \sim \xi$.
We saw in~\cite{LePa14} that, when $f$ is $2R$ times differentiable with $f^{(k)}$ bounded, the \emph{nested} Monte Carlo estimator satisfies \eqref{lemairestrong_error} with $\beta=1$ and \eqref{lemaireweak_error} with $\alpha=1$ and $\bar{R}=R-1$. Here we want to show that the $nested$ Monte Carlo satisfies also the assumptions of the Strong Law of Large Numbers \ref{th::slln} and of the Central Limit Theorem \ref{th::clt_le1}. Then, we define for convenience
\begin{equation}
 \phi_0(y) := \esp{F(\xi,y)}, \quad
 \phi_h(y) := \frac{1}{K} \sum_{k=1}^{K} F(\xi_k,y),\quad K\!\in K_0\N^*,
\end{equation}
so that $Y_0 = f(\phi_0(Y))$ and $Y_h = f(\phi_h(Y))$, and for a fixed $y$, we set $\sigma_F(y) := \sqrt{\Var (F(\xi,y))}$. 

\begin{Proposition} \label{prop::Lp_inequality_nested}
Still assuming that $f$ is Lispchitz continuous. If $X\in L^p(\PP)$ for $p \ge 2$, then there exists $V_1^{(p)}$ such that, for all $h=\frac 1 K$ and $h' = \frac 1 {K'}, K,K' \in K_0\N^*$, 
 \begin{equation} \label{eq::Lp_inequality_nested}
  \normLp{p}{Y_{h'} - Y_{h}}^p \leq V_1^{(p)} \left| h'-h\right|^{\frac p2}.
 \end{equation}
As a consequence, the assumption~\eqref{lemairestrong_error} and the $L^p$--strong error assumption~\eqref{hp::lemairestrong_error_p} are satisfied with $\beta = 1$. Then both ML2R and MLMC estimators satisfy a Strong Law of Large Numbers, see Theorem~\ref{th::slln}.
\end{Proposition}

\begin{proof}
Set $\widetilde X_k = F(\xi_k, Y) - \espc{F(\xi_k,Y)}{Y}$ and $S_k = \sum_{\ell=1}^k \widetilde X_\ell$.
As $f$ is Lipschitz,
\begin{align*}
 \normLp{p}{Y_{h'} - Y_{h}}^p &=  \normLp[4]{p}{f \left( \frac 1 {K'} \sum_{k=1}^{K'} F(\xi_k, Y) \right) - f \left( \frac 1 K \sum_{k=1}^K F(\xi_k, Y) \right)}^p \\
 %&\leq \left[f\right]_{Lip}^p \left\|  \frac 1 {K'} \sum_{k=1}^{K'} F(\xi_k, Y)  - \frac 1 K \sum_{k=1}^K F(\xi_k, Y)  \right\|_p^p \\
 &\leq \left[f\right]_{\Lip}^p \normLp[4]{p}{\frac 1 {K'} \sum_{k=1}^{K'} \widetilde X_k  - \frac 1 K \sum_{k=1}^K \widetilde X_k}^p 
 = \left[f\right]_{\Lip}^p \esp{\left| \frac{S_{K'}}{K'} - \frac{S_K}{K}\right|^p}.
\end{align*}
Assume without loss of generality that $K\leq K'$. Since $p\geq 2$,
\begin{align*}
  \esp{\left| \frac{S_{K'}}{K'} - \frac{S_K}{K}\right|^p} &= \esp{\left| \left( \frac{1}{K'} - \frac{1}{K}\right) S_{K} + \frac{1}{K'} (S_{K'} - S_K)\right|^p} \\
  &\leq 2^{p-1} \left[ \left| \frac{1}{K'} - \frac{1}{K}\right| ^p \esp{|S_{K}|^p} + \left( \frac{1}{K'}\right)^p \esp{\left|S_{K'}-S_K\right|^p} \right].
\end{align*}
Owing to Burkholder's inequality, there exists a universal constant $C_p$ such that
\begin{align*}
\esp{|S_{K}|^p} &\leq C_p \esp{\left|\sum_{k=1}^K  \widetilde X_k^2\right|^{\frac p2}} \leq C_p \left( \sum_{k=1}^K \left\| \widetilde X_k^2 \right\|_{\frac p2}\right)^{\frac p2} = C_p K^{\frac p2} \esp{\left| \widetilde X_1\right|^p}.
%  \esp{|S_{K}|^p} \leq C_p \left[ \left( \sum_{k=1}^K \esp{\widetilde{X}_k^2} \right)^{\frac p2} + \sum_{k=1}^K\esp{|\widetilde{X}_k|^p} \right]  = C_p \left[  K^{\frac p2} \left(\esp{\widetilde{X}_1^2} \right)^{\frac p2} + K \esp{|\widetilde{X}_1|^p} \right].
\end{align*}
Hence, as $S_{K'}-S_K \sim S_{K'-K}$ in distribution, 
\begin{align*}
\esp{\left| \frac{S_{K'}}{K'} - \frac{S_K}{K}\right|^p} 
\leq 2^{p-1}  C_p  \esp{|\widetilde{X}_1|^p}\left[ \left| \frac{1}{K'} - \frac{1}{K}\right|^p K^{\frac p2} + \left( \frac{1}{K'}\right)^p |K'-K|^{\frac p2}\right].
\end{align*}
Keeping in mind that $K'\geq K$, we derive
\begin{equation*}
\left| \frac{1}{K'} - \frac{1}{K}\right|^p K^{\frac p2} + \left| \frac{1}{K'}\right|^p |K'-K|^{\frac p2} 
= \left| \frac{1}{K'} - \frac{1}{K} \right|^{\frac p2}  \left| \frac{K}{K'} - 1  \right|^{\frac p2} +  
\left| \frac{K}{K'}\right| ^{\frac p2}\left| \frac{1}{K} - \frac{1}{K'} \right|^{\frac p2} \leq 2 \left| \frac{1}{K} - \frac{1}{K'} \right|^{\frac p2}.
\end{equation*}
We conclude by setting $V_1^{(p)} = \left[f\right]_{\Lip}^p 2^{p}  C_p \esp{|\widetilde{X}_1|^p}$.
\end{proof}

For the Central Limit Theorem to hold, the key point is the following Lemma.

\begin{Lemma} 
\label{prop::law_convergence}
Assume that $f:\RR\to\RR$ is a Lipschitz continuous function and differentiable with $f'$ continuous. Let $\zeta$ be an ${\cal N}(0,1)$-distributed random variable independent of $Y$. Then, as $h\to0$,
\begin{equation} \label{eq:prop_law_convergence}
    Z(h) = \sqrt{\frac{M}{h}} \pa{Y_{\frac{h}{M}} - Y_h} \xrightarrow{\calL} \sqrt{M-1} f^\prime(\phi_0(Y)) \sigma_F(Y) \zeta.
\end{equation}
%Similarly, $h^{-\frac{1}{2}} (Y_h - Y_0) \xrightarrow{\calL} 
% \begin{align*}
%    &(a) \qquad
%    \frac{Y_h - Y_0}{\sqrt{h}} \xrightarrow{\calL} f^\prime(\phi_0(Y)) \sigma_F(Y) \zeta, \\
%    &(b) \qquad
%    \frac{Y_h - Y_{\frac h M}}{\sqrt{h}} \xrightarrow{\calL} \sqrt{\frac{M-1}{M}} f^\prime(\phi_0(Y)) \sigma_F(Y) \zeta.
%  \end{align*}
\end{Lemma}

\begin{proof}
First note that $Z(h) = z_h^{(M)}(Y)$ where $z_h^{(M)}$ is defined by 
\begin{equation*}
    \forall y \in \R^{q_Y}, \quad z_h^{(M)}(y) = \sqrt{\frac{M}{h}} \pa{f(\phi_{\frac{h}{M}}(y)) - f(\phi_h(y))}.
\end{equation*}
Let $y \in \R^{q_Y}$. We have 
\begin{equation} \label{eq:rep_zhM}
    z_h^{(M)}(y) = - \pa{\int_0^1 f'\pa{v \phi_h(y) + (1-v) \phi_{\frac{h}{M}}(y)} \d v} u^{(M)}_h(y)
\end{equation}
with $u^{(M)}_h(y) = \sqrt{\frac{M}{h}} \pa{\phi_h(y) - \phi_{\frac{h}{M}}(y)}$.
 We derive from the Strong Law of Large Numbers that $\lim_{h \to 0} \phi_h(y) = \phi_0(y) = \lim_{h \to 0} \phi_{\frac{h}{M}}(y)$  $a.s.$ and by continuity of the function $(x_1, x_2) \mapsto \int_0^1 f'(v x_1 + (1-v) x_2) \d v$ (since $f'$ is continuous) we get 
\begin{equation} \label{eq:zhM_first}
    \lim_{h \to 0} \int_0^1 f'\pa{v \phi_h(y) + (1-v) \phi_{\frac hM}(y)} \d v = f'\pa{\phi_0(y)} \quad a.s.
\end{equation}

We have now to study the convergence of the random sequence $u^{(M)}_h(y)$ as $h$ goes to zero. We set $\widetilde{\xi}_k = \xi_{k+K}$, $k=1,\ldots, K(M-1)$. Note that $(\widetilde{\xi}_k)_{k=1, \ldots, K(M-1)}$ are \iid with distribution $\xi_1$ and are independent of $(\xi_k)_{k=1\ldots,M}$. Then we can write
  \begin{align*}
    u^{(M)}_h(y) %\sqrt{\frac{M}{h}} \bigg( \phi_h(y) - \phi_{\frac h M} (y) \bigg) 
&= \sqrt{M K} \left( \frac 1 K \sum_{k=1}^K F(\xi_k, y) - \frac 1 {MK} \sum_{k=1}^{MK} F(\xi_k, y) \right) \\
					&= \sqrt{M K} \left( \frac {M-1}{MK} \sum_{k=1}^K \left(F(\xi_k, y) -\phi_0(y)\right)- \frac 1 {MK} \sum_{k=K+1}^{MK} \left(F(\xi_k, y)-\phi_0(y)\right) \right) \\
					&= \frac{M-1}{\sqrt{M}} \left( \frac {1}{\sqrt{K}} \sum_{k=1}^K F(\xi_k, y) - \phi_0(y) \right) \\					
& \qquad \qquad 
- \sqrt{\frac{M-1}{M}} \left[  \frac{1}{\sqrt{K(M-1)}} \left( \sum_{k=1}^{K(M-1)} F(\widetilde{\xi}_k, y) - \phi_0(y) \right) \right].
  \end{align*}
  Owing to the Central Limit Theorem and the independence of both terms in the right hand side of the above inequality, we derive that 
$$ u^{(M)}_h(y)
%\frac 1 {\sqrt{h}} \bigg( \phi_h(y) - \phi_{\frac h M} (y) \bigg) 
\xrightarrow{\mathcal{L}} \frac{M-1}{\sqrt{M}} \sigma_F(y) \zeta_1 - \sqrt{\frac{M-1}{M}} \sigma_F(y) \zeta_2, \quad \mbox{as} \; h\to0,$$
  where $\zeta_1$ and $\zeta_2$ are two independent random variables  both following a standard Gaussian distribution. Hence, noting that $\left( \frac{M-1}{\sqrt{M}}\right)^2 + \left( \sqrt{\frac{M-1}{M}} \right)^2 = M-1$, we obtain
  \begin{equation} \label{eq:uhM_conv}
u^{(M)}_h(y)
%   \frac 1 {\sqrt{h}} \bigg( \phi_h(y) - \phi_{\frac h M} (y) \bigg) 
\xrightarrow{\mathcal{L}} \sqrt{M-1} \sigma_F(y) \zeta \quad \mbox{with} \quad \zeta\sim\mathcal{N}(0,1).
  \end{equation}

  By Slutsky's Theorem, we derive from \eqref{eq:rep_zhM}, \eqref{eq:zhM_first} and \eqref{eq:uhM_conv} that for every $y\in\RR^{q_Y}$,
\begin{equation}
    z_h^{(M)}(y) \xrightarrow{\calL} \sqrt{M-1} f'\pa{\phi_0(y)} \sigma_F(y) \zeta, \quad \text{as $h \to 0$}.  
    \label{eq::h_M_convergence}
%    f^\prime \left(\phi_{\frac h M}(y)\right) \frac {\phi_h(y) - \phi_{\frac h M}(y)}{\sqrt{h}} \xrightarrow{\calL} f^{\prime} (\phi_0(y)) \sqrt{\frac{M-1}{M}} \sigma_F(y) \zeta.
\end{equation}

Recall that $Z(h) = z_h^{(M)}(Y)$. We prove \eqref{eq:prop_law_convergence} combining Fubini's theorem with Lebesgue dominated convergence theorem and \eqref{eq::h_M_convergence}. More precisely, for all $G \in \mathcal{C}_b$ we have 
\begin{align*}
    \lim_{h \to 0} \esp[2]{G(Z(h))} &= \lim_{h \to 0} \esp{G\pa{z_h^{(M)}(Y)}} = \esp{\lim_{h \to 0} G\pa{z_h^{(M)}(Y)}} \\ 
    & = \esp{G\pa{\sqrt{M-1} f'\pa{\phi_0(y)} \sigma_F(y) \zeta}}.
\end{align*}
\end{proof}

We are now in position to prove that the \emph{nested} Monte Carlo satisfies the assumptions of the Central Limit Theorem \ref{th::clt_le1}.

\begin{Theorem}
 Assume that $f:\RR\to\RR$ is a Lipschitz continuous function and differentiable with $f'$ continuous. Then $\pa[1]{Z(h)}_{h \in \Hr}$ is $L^2$--uniformly integrable and 
\begin{equation} \label{eq:thm_final_nested}
    \lim_{h \to 0} \normLp{2}{Z(h)}^2 = (M-1) \normLp{2}{f^\prime(\phi_0(Y)) \sigma_F(Y)}^2
\end{equation}
 As a consequence, the ML2R and MLMC estimators \eqref{intro:MLRR-estimator} and \eqref{intro:MLMC-estimator} satisfy a Central Limit Theorem in the sense of Theorem~\ref{th::clt_le1} (case $\beta =1$).
\end{Theorem}

\begin{proof}
We prove first the $L^2$--uniform integrability of $\pa[1]{Z(h)}_{h \in \Hr}$. As $f$ is Lipschitz we have, 
\begin{equation*}
    \abs{Z(h)}^2 \le [f]_{\Lip}^2 \abs{u^{(M)}_h(Y)}^2, \quad \text{with} \quad u^{(M)}_h(y) = \sqrt{\frac{M}{h}} \pa{\phi_h(y) - \phi_{\frac{h}{M}}(y)}.
\end{equation*}
Consequently it suffices to show that $\pa{u^{(M)}_h(Y)}_{h \in \Hr}$ is $L^2$--uniformly integrable, to establish the $L^2$--uniform integrability of $\pa[1]{Z(h)}_{h \in \Hr}$.

We saw in the proof of Proposition \ref{prop::law_convergence} that $u^{(M)}_h(Y) \xrightarrow{\mathcal L} \sqrt{M-1} \sigma_F(Y) \zeta$ as $h$ goes to $0$, where $\zeta$ is a standard normal random variable independent of $Y$. Owing to Lemma~\ref{prop::L_2_norm_unif_integrability} $(b)$, the uniform integrability will follow from $\lim_{h \to 0} \normLp{2}{u^{(M)}_h(Y)} = \normLp{2}{\sqrt{M-1} \sigma_F(Y) \zeta}$. In fact this convergence holds as an equality. Indeed
\begin{equation*}
 \normLp{2}{u^{(M)}_h(Y)}^2 = MK \esp{\left( S_K - S_{MK}\right)^2} = MK \esp{\left( \frac{M-1}{MK} S_K -\frac{1}{MK}(S_{MK} - S_K)\right)^2}.
\end{equation*}
We notice that $S_{MK} - S_K$ is independent of $S_K$. Hence, since the $\xi_k$ are independent,
\begin{align*}
 \normLp{2}{u^{(M)}_h(Y)}^2 &= MK \left( \esp{\left(\frac{M-1}{MK} S_K\right)^2} + \esp{\left(\frac{1}{MK}(S_{MK} - S_K)\right)^2} \right) \\
 &= MK \left( \left(\frac{M-1}{MK}\right)^2 K \esp{(\widetilde X_1)^2} + \left(\frac{1}{MK}\right)^2 (MK-K)\esp{\left(\widetilde X_1\right)^2} \right) \\
 &= (M-1) \esp{(\widetilde X_1)^2} = (M-1) \esp{\sigma^2_F(Y)}.
\end{align*}

We prove now~\eqref{eq:thm_final_nested} using again the Lemma~\ref{prop::L_2_norm_unif_integrability} $(b)$ with the convergence in law of $\pa[1]{Z(h)}_{h \in \Hr}$ established in Lemma~\ref{prop::law_convergence}.
\end{proof}

We notice that, if the assumption \eqref{eq::Lp_inequality_nested} in Proposition~\ref{prop::Lp_inequality_nested} holds with $p>2$, the condition of $L^2$--uniform integrability is much easier to show since it is a direct consequence of Lemma \ref{prop::L_2_norm_unif_integrability} $(a)$. 
% this condition leads to the $L^2$--uniform integrability of the family $(Z_j)_{j\geq1}$, hence, to satisfy the assumptions of the Central Limit Theorem in the case $\beta=1$ for $p>2$, it only remains to show the sharp \eqref{lemairestrong_error} assumption \eqref{eq::norm_L_2_convergence}. This follows from Lemma \eqref{prop::L_2_norm_unif_integrability} $(b)$, owing to the convergence in law of $h^{-\frac 12} \left( Y_{\frac hM} - Y_h\right)$, proved in the Proposition \ref{prop::law_convergence}.

\subsection{Smooth nested Monte Carlo}
When the function $f $ is smooth, namely ${\cal C}^{1+\rho}(\R,\R)$, $\rho\!\in (0,1]$  ($f'$ is $\rho$-H\"older), a variant of the former multilevel nested estimator has been used in~\cite{BuHaRe15} (see also~\cite{Gi15}) to improve the strong rate of convergence in order to attain the asymptotically unbiased setting $\beta>1$ in the condition~\eqref{lemairestrong_error}. A root $M$ being given, the idea is  to replace in the successive refined levels the difference $Y_{\frac hM}-Y_h$ (where $h= \frac{1}{K}$, $K\!\in K_0\N^*$)  in the ML2R et MLMC estimators by 
\[
Y_{h,\frac hM} :=  f\left(\frac{1}{MK} \sum_{k=1}^{MK} F\big(\xi_{k},Y\big)\right)- \frac 1M \sum_{m=1}^M f\left(\frac 1K \sum_{k=1}^K F\big(\xi_{(m-1)K+k},Y\big)\right) .
\]
It is clear that $\esp{Y_{h,\frac hM} } = \esp{Y_{\frac hM}-Y_{h} }$. Computations similar to those carried out in Proposition~\ref{prop::Lp_inequality_nested} yield that,  if $X = F(\xi,Y)\!\in L^{p(1+\rho)}(\PP)$ for some $p\ge 2$, then
\begin{equation}\label{eq:Y-(h,h/M)}
\big\|Y_{h,\frac hM}\big\|_p^p \le V^{(\rho,p)}_M\Big|h-\frac h M\Big|^{\frac p2(1+\rho)}= V^{(\rho, p)}_M\Big|1-\frac 1 M\Big|^{\frac p2(1+\rho)}|h|^{\frac p2(1+\rho)}.
\end{equation}

\noindent $\rhd$ {\it SLLN}: The first consequence is that the SLLN also holds for these modified estimators along the sequences of RMSE $(\varepsilon_k)_{k\ge 1}$ satisfying $\sum_{k\ge 1}\varepsilon_k^p <+\infty$ owing to Theorem~\ref{th::slln}.

\medskip
\noindent $\rhd$ {\it CLT}: When~\eqref{eq:Y-(h,h/M)}  is satisfied with $p=2$, one derives that  $\beta = \frac p2(1+\rho) = 1+\rho>1$ whatever $\rho$ is. Hence, the only requested condition in this setting to obtain a CLT (see~Theorem~\ref{th::clt_g1}) is the $L^2$--uniform integrability of $\pa[1]{h^{-\frac\beta2}Y_{h,\frac hM}}_{h \in \Hr}$, since no sharp rate is needed when $\beta >1$. Moreover, if~\eqref{eq:Y-(h,h/M)} holds for  a $p\in (2,+\infty)$, $i.e.$ if $X = F(\xi,Y)\!\in L^{p(1+\rho)}(\PP)$ with $p>2$, then $h^{-\frac {\beta}{2}}\big\| Y_{h,\frac hM}\big\|_p \le  {V^{(\rho, p)}_M}^{\frac 1p}\Big|1-\frac 1 M\Big|^{\frac 12(1+\rho)}$ which in turn ensures the  $L^2$--uniform integrability.

As a final remark, note that if the function $f$ {\em is  convex,} $Y_{h,\frac hM} \le 0$ so that $ \esp{Y_{\frac hM} }\le \esp{Y_{h} }$ which in turn implies by an easy induction that $\esp{Y_0}\le \esp{Y_h}$ for every $h\!\in \Hr$. A noticeable consequence is that the MLMC estimator   has a positive bias.

\smallskip 
These  results  can be extended to locally $\rho$-H\" older continuous functions with polynomial growth at infinity. For more details and a complete proof we refer to~\cite{Gi17}. 

\section{Asymptotic of the  weights}
\label{appendix::weights}

We focus our attention on the behaviour of  $\W_j^R$ when $R\to+\infty$.
We recall
$$\W_j^R = \sum_{\ell=j}^R a_\ell b_{R-\ell} = \sum_{\ell=0}^{R-j} a_{R-\ell} b_{\ell}$$
with
$$a_\ell = \frac{1}{\prod_{1\leq k \leq \ell-1} (1- M^{-k\alpha}) }$$
and with the convention $\prod_{k=1}^0 (1- M^{-k\alpha}) = 1$, and 

$$b_{\ell} = (-1)^{\ell} \frac{M^{-\frac\alpha 2 \ell(\ell+1)}}{\prod_{1\leq k \leq \ell} (1- M^{-k\alpha})}.$$
For convenience, we set $\W_j^R = 0$, for $j\geq R+1, R\in \N^*$.
We first notice that $a_\ell$ is an increasing and converging sequence and we set 

$$\lim_{\ell\to +\infty} a_\ell = a_\infty.$$
The sequence $b_\ell$ converges to zero
and furthermore the series with general term $b_\ell$ is absolutely converging, since $\sum_{\ell\geq1} M^{-\frac{\alpha}{2}\ell(\ell+1)} < +\infty$. This leads us to set

$$\widetilde{B}_\infty  = \sum_{\ell=0}^{+\infty} |b_\ell| < +\infty
\quad \mbox{and} \quad 
B_\infty = \sum_{\ell=0}^{+\infty} b_\ell  < +\infty.$$
Claim $(a)$ of Lemma \ref{lemma::W_j_R_limits} is then proved.
As a consequence, 
\begin{equation}
 \forall R \in \N^*, \forall j\in \left\{ 1, \ldots, R\right\}, \qquad \left| \W_j^R\right| \leq a_\infty \widetilde{B}_\infty,
\end{equation}
which proves claim $(b)$ in Lemma \ref{lemma::W_j_R_limits}. For the proof of claims $(c)$ and $(d)$, 
we will need the following

\begin{Lemma}
\label{lemma::W_j_R_convergence}
  Let $\varphi:\N^*\to\N^*$ such that $\varphi(R) \in \{1, \ldots,R-1\}$ for every $R\geq1$, 
  $\varphi(R) \to +\infty$ and $R-\varphi(R) \to+\infty$ as $R\to+\infty$. Then
  $$\lim_{R\to +\infty}\sup_{1\leq j\leq\varphi(R)} \left| \W_j^R - 1  \right| = 0. $$
  In particular, $\forall j \in \N^*, \W_j^R\to1$ as $R\to+\infty$.\\
  However, this convergence is not uniform since $\W_{R-j+1}^R \to a_\infty \sum_{\ell=0}^{j-1} b_\ell$ for every $j\in\N^*$ as $R\to+\infty$.
\end{Lemma}

\begin{proof}
  We write
  
  \begin{align*}
   \left| \W_j^R- a_\infty B_\infty  \right| & = \left| \sum_{\ell=0}^{R-j} a_{R-\ell} b_\ell - a_\infty \sum_{\ell = 0}^{R-j} b_\ell - a_\infty \sum_{\ell \geq R-j+1} b_\ell\right|\\
				            & \leq \sum_{\ell=0}^{R-j} \left( a_\infty - a_{R-\ell} \right)|b_\ell| + a_\infty \sum_{\ell\geq R-j+1} |b_\ell|
  \end{align*}
  First note that
  $$\lim_{R\to+\infty} \sup_{j\in\left\{ 1,\ldots,\varphi(R) \right\}} \sum_{\ell\geq R-j+1} |b_\ell| \leq \lim_{R\to+\infty}\sum_{\ell\geq R-\varphi(R)+1} |b_\ell|= 0,$$
  as $R-\varphi(R) \to +\infty$ and $\sum_{\ell \geq0} |b_\ell| < +\infty$. On the other hand, for every $j\in\left\{ 1, \ldots,\varphi(R)\right\}$,
  
  \begin{align*}
   \sum_{\ell=0}^{R-j} \left( a_\infty - a_{R-\ell} \right)|b_\ell| & =  \sum_{\ell=j}^{R} \left( a_\infty - a_{\ell} \right)|b_{R-\ell}| \\
   & = \sum_{\ell=j}^{\varphi(R)} \left( a_\infty - a_{\ell} \right)|b_{R-\ell}| + \sum_{\ell=\varphi(R)+1}^{R} \left( a_\infty - a_{\ell} \right)|b_{R-\ell}|\\
   & \leq a_\infty  \sum_{\ell=R-\varphi(R)}^{R-j}|b_{\ell}| + \left( a_\infty - a_{\varphi(R)+1}\right) \sum_{\ell=0}^{R-\varphi(R)-1} |b_\ell|\\
   & \leq a_\infty  \sum_{\ell=R-\varphi(R)}^{+\infty}|b_{\ell}| + \left( a_\infty - a_{\varphi(R)+1}\right) \sum_{\ell=0}^{R-\varphi(R)-1} |b_\ell|.
  \end{align*}
  Consequently, $\sup_{j\in \left\{ 1,\ldots,R \right\}} \sum_{\ell=0}^{R-j} \left( a_\infty - a_{R-\ell}\right)|b_\ell| \to 0$ as $R\to+\infty$, 
  since $\varphi(R)$ and $R-\varphi(R) \to+\infty$. Finally,  
  
  $$\lim_{R\to+\infty} \sup_{1\leq j\leq\varphi(R)} \left| \W_j^R - a_\infty B_\infty  \right| = 0.$$
  Moreover, by definition $\W_1^R = 1$ for all $R$, which implies that $B_\infty = \frac 1{a_\infty}$ and completes the proof.
  Finally, as $a_j \to a_\infty$,

  $$\W_{R-j+1}^{R} = \sum_{\ell=0}^{j-1} a_{R-\ell}b_l \xrightarrow{R\to+\infty} a_\infty \sum_{\ell=0}^{j-1} b_\ell.$$

\end{proof}

\begin{proof}[Proof of Lemma \ref{lemma::W_j_R_limits} $(c)$ and $(d)$] ~ 
 \begin{itemize}
  \item[$(c)$] Let us consider the non-negative measure on $\N^*$ defined by $m_\beta(j) = M^{\gamma(j-1)}$, $\gamma<0$. We notice that it is a finite measure since 
    $$ \sum_{j\geq 1} dm_\beta(j) = \frac{1}{1-M^{\gamma}}.$$
    Since, as we saw in Lemma \ref{lemma::W_j_R_convergence}, $\W_j^R \to1$ as $R\to+\infty$ for every $j\in\N^*$ and $|\W_j^R|\leq a_\infty \widetilde{B}_\infty$, 
    we derive from Lebesgue's dominated convergence theorem that
    \begin{equation*}
      \lim_{R\to+\infty} \sum_{j=2}^R \left| \W_j^R \right| M^{\gamma(j-1)} = \sum_{j=2}^{+\infty} \lim_{R\to+\infty} \left| \W_j^R \right| M^{\gamma(j-1)} = \frac{1}{M^{-\gamma}-1}.
    \end{equation*}
    
  \item[$(d)$] $\bullet$ If $\gamma<0$, we consider the non-negative finite measure on $\N^*$ defined by $m'_\beta(j) = M^{\gamma(j-1)} v_j$ since $(v_j)_{j \ge 1}$ is a bounded sequence of positive real numbers. As in the previous case $(c)$ we have  
    \begin{equation*}
      \lim_{R\to+\infty} \sum_{j=2}^R \left| \W_j^R \right| M^{\gamma(j-1)} v_j = \sum_{j=2}^{+\infty} M^{\gamma(j-1)} v_j.
    \end{equation*}

  $\bullet$ If $\gamma=0$, let us consider a sequence $\varphi(R) \in \{1, \ldots,R-1\}$ such that 
  $\frac{\varphi(R)}{R} \to 1$, $R-\varphi(R) \to+\infty$  as $R\to+\infty$ (for example $\varphi(R) = R-\sqrt{R}$). Then we can write
  
  \begin{align*}
  \left| \frac 1 {R} \sum_{j=2}^{R} \left| \W_j^{R}\right| v_j -\frac 1 {R} \sum_{j=2}^{R}  v_j \right| 
  &\leq \left[ \frac 1 R \sum_{j=2}^{\varphi(R)} \left| \left| \W_j^R \right| -1\right| + \frac 1 R \sum_{j=\varphi(R)+1}^R (|\W_j^R|+1)  \right] \sup_{j\geq2} v_j\\
  &\leq \left[ \sup_{2\leq j\leq \varphi(R)} \left|  \W_j^{R} -1\right| \frac{\varphi(R)}{R} + \left( a_\infty \widetilde{B}_\infty  +1 \right) \left( 1 - \frac{\varphi(R)}{R}\right) \right] \sup_{j\geq 2} v_j .
  \end{align*}
  Owing to Lemma \ref{lemma::W_j_R_convergence} $\sup_{2\leq j\leq \varphi(R)} |\W_j^R -1| \to 0$ as $R\to+\infty$. 
  Using furthermore that $\frac{\varphi(R)}{R} \to 1$ as $R\to+\infty$ and that $\lim_{j \to +\infty} v_j = 1$, one concludes by noting that, owing to C\'es\`aro's Lemma, 
  $\ds \lim_{R\to+\infty} \frac 1 {R} \sum_{j=2}^{R} v_j = 1$.

  $\bullet$ If $\gamma>0$, first, we notice that

  \begin{equation}
  \label{eq::to_infty}
   \sum_{j=2}^{R} \left| \W_j^{R}\right| M^{\gamma(j-1)} \geq \left| \W_{R}^{R} \right|  M^{\gamma R} = \left| a_R \right|  M^{\gamma R} \to +\infty.
  \end{equation}
  Let $\eta>0$. Since $\lim_{j\to+\infty}v_j=1$, there exists $N_\eta\in\N^*$ such that, for each $j>N_\eta$, 
  $\left| v_j-1\right|<\frac{\eta}{2}$. Owing to Lemma \ref{lemma::W_j_R_convergence} there exists $R_\eta$ such that,
  for each $R \geq R_\eta$, $\sup_{2\leq j\leq N_\eta} \left| \W_j^{R}\right|< 1 +\eta$. Then, 

  \begin{align*}
  \left| \frac{\sum_{j=2}^{R} \left| \W_j^{R}\right| M^{\gamma(j-1)} v_j}{\sum_{j=2}^{R} \left| \W_j^{R}\right| M^{\gamma(j-1)}} -1 \right| &\leq \frac{\sum_{j=2}^{R} \left| \W_j^{R}\right| M^{\gamma(j-1)} \left|v_j-1\right|}{\sum_{j=2}^{R} \left| \W_j^{R}\right| M^{\gamma(j-1)}} \\
  &\leq \frac{\sum_{j=2}^{N_\eta} \left| \W_j^{R}\right| M^{\gamma(j-1)} \left|v_j-1\right|}{\sum_{j=2}^{R} \left| \W_j^{R}\right| M^{\gamma(j-1)}} + \frac \eta 2 \frac{\sum_{j=N_\eta+1}^{R} \left| \W_j^{R}\right| M^{\gamma(j-1)} }{\sum_{j=2}^{R} \left| \W_j^{R}\right| M^{\gamma(j-1)}} \\
  &\leq \frac{\max_{2\leq j \leq N_\eta} M^{\gamma(j-1)} \left|v_j-1\right| N_\eta \sup_{2\leq j \leq N_\eta} \left| \W_j^{R}\right| }{\sum_{j=2}^{R} \left| \W_j^{R}\right| M^{\gamma(j-1)}} + \frac \eta 2\\
  &\leq \frac{f(N_\eta) (1+\eta)}{\sum_{j=2}^{R} \left| \W_j^{R}\right| M^{\gamma(j-1)}} + \frac \eta 2
  \end{align*}
  where $f(N) = \max_{2\leq j \leq N} M^{\gamma(j-1)} \left|v_j-1\right| N$ does not depend on $R$.
  Thanks to \eqref{eq::to_infty}, there exists $R'_\eta>0$ such that, for each $R \geq \max(R_\eta, R'_\eta)$,  
  $\sum_{j=2}^{R} \left| \W_j^{R}\right| M^{\gamma(j-1)}> \frac{2f(N_\eta)(1+\eta)}{\eta}$, which proves that 
  
  $$\lim_{R\to+\infty} \frac{\sum_{j=2}^{R} \left| \W_j^{R}\right| M^{\gamma(j-1)} v_j}{\sum_{j=2}^{R} \left| \W_j^{R}\right| M^{\gamma(j-1)}} =1.$$
  This leads to analyze 
  
  \begin{align*}
  \frac{1}{M^{\gamma R}} \sum_{j=2}^{R} \left| \W_j^{R} \right| M^{\gamma(j-1)} = \sum_{j=2}^{R} \left| \W_j^{R} \right| M^{-\gamma(R-j+1)} = \sum_{j=1}^{R-1} \left| \W_{R-j+1}^{R} \right| M^{-\gamma j}.
  \end{align*}
  Using that $\left| \W_j^{R} \right| \leq a_\infty \widetilde{B}_\infty $ for $j\in \left\{ 1, \ldots, R\right\}$ 
  and Lemma \ref{lemma::W_j_R_convergence} one derives from Lebesgue's dominated convergence theorem that 

  $$\sum_{j=1}^{R-1} \left| \W_{R-j+1}^{R} \right| M^{-\gamma j} \xrightarrow{R\to+\infty} a_\infty \sum_{j\geq 1} \left| \sum_{\ell=0}^{j-1} b_\ell \right| M^{-\gamma j}<+\infty$$
  since $\left| \sum_{\ell=0}^{j-1} b_\ell \right| \leq \sum_{\ell=0}^{j-1} \left| b_\ell \right| \leq \widetilde{B}_\infty$.
  
 \end{itemize}
 
\end{proof}

\section{Additional computations for Proposition \ref{prop::L_p_convergence} in the MLMC case}
\label{appendix::slln_mlmc_computations}
Here below we give the computations needed to prove inequality \eqref{eq::mlmc_ssln_inequality} in the proof of Proposition \ref{prop::L_p_convergence} for the MLMC estimator.

  \begin{itemize}
	  \item[$\bullet$] For  $p<\beta+1$ (so that $\beta >1$): As we did with the ML2R estimator,
		$$\frac{1}{N(\varepsilon)^{p-1}} \sum_{j=2}^{R(\varepsilon)}  M^{\frac{p - (\beta+1)}{2}(j-1)} \leq K_2 \varepsilon^{2(p-1)} \leq K_2 \varepsilon^p$$
		with $K_2 = K_2(M,\beta,p) = 1/\left( 1-M^{\frac{p-(\beta+1)}{2}}\right)  (2/C_\beta)^{p-1}$, since $p\geq 2$.
		  
	  \item[$\bullet$] For $p=\beta+1$: 
		  
		$$\frac{1}{N(\varepsilon)^{p-1}} \sum_{j=2}^{R(\varepsilon)}  M^{\frac{p - (\beta+1)}{2}(j-1)} \leq \frac{R(\varepsilon)}{N(\varepsilon)^{p-1}} .$$
		  
		Since $p\geq2$, we have $\beta\geq1$. The case $\beta=1$ leads to $p=2$, which is a consequence of \eqref{eq::L2_convergence}.
		If $\beta >1$, we derive from \eqref{eq::R_MLMC},
		  
		$$\frac{R(\varepsilon)}{N(\varepsilon)^{p-1}} \leq \left(C_R^{(1)}+1\right) \left( \frac 2 {C_\beta}\right)^{p-1} \varepsilon^{2(p-1)} + \frac 1 {\alpha \log(M)} \left( \frac 2 {C_\beta}\right)^{p-1} \varepsilon^{2(p-1)} \log\left( \frac 1 \varepsilon\right).$$
		Since $p>2$, we have $0< \varepsilon^{2(p-1)} \leq \varepsilon^{2(p-1)} \log\left( \frac 1\varepsilon\right) \leq \varepsilon^p \sup_{\varepsilon\in(0,1\wedge\bar\varepsilon]} \left( \varepsilon^{p-2} \log\left(\frac 1 \varepsilon\right) \right) $. 
		Then
		$$\frac{R(\varepsilon)}{N(\varepsilon)^{p-1}} \leq K_2 \varepsilon^p$$
		with $K_2 = K_2(M,\beta,p) = \left( \frac 2 {C_\beta}\right)^{p-1}  \left(C_R^{(1)} + 1 + \frac{1}{\alpha\log(M)}\right) \sup_{\varepsilon\in(0,1\wedge\bar\varepsilon]} \left( \varepsilon^{p-2} \log\left(\frac 1 \varepsilon\right) \right)$.
		  
	  \item[$\bullet$] For $p > \beta + 1$: As $p-(\beta+1) >0$, one has
		  
		$$\frac{1}{N(\varepsilon)^{p-1}} \sum_{j=2}^{R(\varepsilon)}  M^{\frac{p - (\beta+1)}{2}(j-1)} \leq \frac{1}{M^{\frac{p-(\beta +1)}{2}}-1} \frac{M^{\frac{p - (\beta+1)}{2}R(\varepsilon)}}{N(\varepsilon)^{p-1}}. $$
		We derive from \eqref{eq::R_MLMC},
		$$M^{\frac{p - (\beta+1)}{2}R(\varepsilon)} \leq M^{\frac{p - (\beta+1)}{2}\left( C_R^{(1)} +1 +\frac{1}{\alpha \log(M)} \log\left( \frac1\varepsilon\right)\right)}
		=M^{\frac{p - (\beta+1)}{2}\left( C_R^{(1)} +1\right)} \varepsilon^{ -\frac{p - (\beta+1)}{2\alpha } }.$$
		On the other hand, owing to \eqref{eq::1_on_N},
		  
		\begin{equation*}
		  \frac{1}{N(\varepsilon)^{p-1}}                                                                                      \leq \left(\frac{2}{C_\beta} \right)^{p-1} \varepsilon^{2(p-1)}  
		  \begin{cases}
		  1                                                                                                                  & \mbox{for } \beta\geq 1,                                        \\
		  M^{-\frac{1-\beta}{2}C_R^{(1)}(p-1)} \varepsilon^{-\frac{1-\beta}{2\alpha}(p-1)}                                   & \mbox{for } \beta<1.                                            
		  \end{cases}
		\end{equation*}
		Collecting these estimates finally yields
		$$\frac{M^{\frac{p - (\beta+1)}{2}R(\varepsilon)}}{N(\varepsilon)^{p-1}} 
		\leq A(\beta, M, p)\varepsilon^{2(p-1)}
		\begin{cases}
		  \varepsilon^{ -\frac{p - (\beta+1)}{2\alpha } }           & \mbox{for } \beta\geq 1, \\
		  \varepsilon^{-\frac\beta\alpha \left( \frac p2 -1\right)} & \mbox{for } \beta<1     
		\end{cases}
		$$
		with
		$
		A(\beta, M, p) = M^{\frac{p - (\beta+1)}{2}\left( C_R^{(1)} +1\right)} \left(\frac{2}{C_\beta} \right)^{p-1}
		\begin{cases}
		  1                                    & \mbox{for } \beta\geq 1, \\
		  M^{-\frac{1-\beta}{2}C_R^{(1)}(p-1)} & \mbox{for } \beta<1.     
		\end{cases}
		$
		\begin{itemize}
		  \item[-] For $\beta\geq1$, note that 
			$$2(p-1)-\frac{p-(\beta-1)}{2\alpha} = p+\frac{(p-2)(2\alpha-1)+\beta-1}{2\alpha}\geq p$$
			since $p\geq2$ and $2\alpha-1\geq0$, since $2\alpha\geq\beta\geq1$. 
		  \item[-] For $\beta<1$, note that
			$$2(p-1)-\frac\beta\alpha\left( \frac p2 -1\right) = p + \frac{(p-2)(\alpha-\beta/2)}{\alpha}\geq p$$
			since $p\geq2$ and $\alpha\geq\frac \beta2$.
		\end{itemize}
		Finally $\displaystyle \frac{1}{N(\varepsilon)^{p-1}} \sum_{j=2}^{R(\varepsilon)}  M^{\frac{p - (\beta+1)}{2}(j-1)} \leq K_2 \varepsilon^p$
		with $\displaystyle K_2 = K_2(M,\beta,p) = \frac{1}{M^{\frac{p-(\beta +1)}{2}}-1} A(\beta, M, p)$.
  \end{itemize}

\printbibliography

\end{document}